\documentclass[twoside,12pt]{article}

\usepackage[T1]{fontenc}
\usepackage[utf8]{inputenc}
\usepackage[english]{babel}
\usepackage[a4paper,margin=1in]{geometry}

\usepackage{amsmath,amssymb,amsfonts,amsthm}
\usepackage{mathtools}
\usepackage{bm}
\usepackage{mathrsfs}
\usepackage{graphicx}
\usepackage{xcolor}
\usepackage{enumitem}
\usepackage{array}
\usepackage{rotating}
\usepackage{authblk}
\usepackage{appendix}
\usepackage{subcaption}

\usepackage{natbib}
\usepackage[
  colorlinks=true,
  linkcolor=blue,
  urlcolor=blue,
  citecolor=blue
]{hyperref}

\allowdisplaybreaks

\newtheorem{thm}{Theorem}[section]
\newtheorem{prp}[thm]{Proposition}

\newcommand{\R}{\mathbb{R}}
\newcommand{\C}{\mathbb{C}}
\newcommand{\N}{\mathbb{N}}
\newcommand{\ud}{\mathrm{d}}
\newcommand{\ind}{\mathbf{1}}
\newcommand{\BR}{\mathscr{B}(\mathbb{R})}
\newcommand{\G}{\mathsf{G}}
\newcommand{\ptilde}{\widetilde{\mathfrak{p}}}
\newcommand{\PP}{\mathbb{P}}
\newcommand{\EE}{\mathbb{E}}
\renewcommand{\L}{\mathfrak{L}}
\newcommand{\m}{\mathfrak{m}}
\newcommand{\V}{\mathfrak{V}}
\newcommand{\CC}{\mathfrak{C}}
\newcommand{\DD}{\mathfrak{D}}
\newcommand{\M}{\mathfrak{M}}

\DeclareMathOperator{\Arg}{Arg}

\providecommand{\keywords}[1]
{
  \small	
  \textbf{\textit{Keywords:}} #1
}

\begin{document}

\title{The variance of the Pitman--Yor process: Cifarelli--Regazzini-type identities and inversion formulas}

\author[1]{Emanuele Dolera\thanks{emanuele.dolera@unipv.it}}
\author[2]{Stefano Favaro\thanks{stefano.favaro@unito.it}}
\affil[1]{\small{Dipartimento di Matematica ``Felice Casorati'', Universit\`a di Pavia, Italia}}
\affil[2]{\small{Dipartimento di Scienze Economico-Sociali e Matematico-Statistiche, Universit\`a di Torino e Collegio Carlo Alberto, Italia}}

\date{\itshape Dedicated to Eugenio Regazzini on the occasion of his 80th birthday}

\maketitle

\begin{abstract}
The Cifarelli--Regazzini identity lies at the foundation of the theory of the Dirichlet process and has played a central role in Bayesian nonparametric statistics. By providing the generalized Cauchy--Stieltjes transform of linear functionals of the Dirichlet process, together with an analytic inversion, it yields explicit representations for the distribution function and density of the Dirichlet mean. This transform identity was subsequently extended to the Pitman--Yor process, a generalization of the Dirichlet process, leading to an analogous distributional theory for its linear functionals. Nonlinear functionals, by contrast, remain less understood, with results largely confined to a few examples for the Dirichlet process. In this paper, we move beyond linear functionals by considering the variance of the Pitman--Yor process, a genuinely quadratic functional. In particular, we establish a Cifarelli--Regazzini-type identity providing the generalized Cauchy--Stieltjes transform of the Pitman--Yor variance for general base probability measures under suitable integrability assumptions. We then analytically invert this transform under general conditions on the base probability measure, thus obtaining explicit representations for both the distribution function and density. Corresponding formulas are also obtained for the Dirichlet variance, yielding more explicit representations than those currently available in the literature. We also establish a distributional representation of the Pitman--Yor variance in terms of a Dirichlet process whose base probability measure is the joint distribution of the mean and variance of a suitably specified Pitman--Yor process, mirroring the classical within--between decomposition of variance and providing a quadratic counterpart to a known distributional identity for Pitman--Yor means.
\end{abstract}



\tableofcontents


\section{Introduction}\label{sec1}

Since the introduction of the Dirichlet process by Thomas S. Ferguson \citep{Fer(73)}, random probability measures have played a central role in Bayesian nonparametric statistics, providing a flexible framework for placing prior distributions on unknown probability laws. Beyond their role as nonparametric priors, functionals of random probability measures may themselves be of substantive inferential interest, including measures of location, dispersion, concentration, and diversity. Among these, linear functionals have received by far the greatest attention in the statistical and probabilistic literature. Their study has given rise to an elegant distributional theory based on exact transform identities and analytic inversion, initiated by Donato M. Cifarelli and Eugenio Regazzini in two pioneering papers published in 1979 \citep{CifReg(79a),CifReg(79b)} and subsequently developed systematically in their landmark contribution \citep{CifReg(90)}.

Building on these foundations, subsequent research has broadened the distributional analysis of linear functionals of Dirichlet processes. \citet{RegGuDN(02)} extended the theory of Dirichlet means beyond the support restrictions imposed in earlier work, whereas \citet{LijReg(04)} established connections with Lauricella hypergeometric functions, deriving symmetry conditions and new analytic representations; see also \citet{Reg(98)}. Further developments include \citet{DiaKem(96)}, who highlighted several mathematical settings related to Dirichlet means, the functional beta--gamma calculus of \citet{Jam(05)}, the connections with generalized gamma convolutions studied by \citet{JamRoyYor(08)}, and the Dirichlet mean identities and Lamperti-type laws investigated by \citet{Jam(10a),Jam(10b)}. In a different direction, \citet{Tsi(99)} extended the Cifarelli--Regazzini identity to the Pitman--Yor process, which generalizes the Dirichlet process \citep{PerPitYor(92),PitYor(97)}; see also \citet{KeTsi(04),VerYorTsi(04)}. The scope of the theory was further broadened by \citet{RegLijPru(03)}, who derived the distribution of
linear functionals of random probability measures obtained by normalizing completely random measures, and by \citet{Epi(03)}, who investigated exponential functionals of neutral-to-the-right priors. We refer to \citet{LijPru(09)} for a comprehensive review of these developments. More recently, \citet{Gaf(25)} considered the inverse problem of determining which specifications of Dirichlet and Pitman--Yor processes give rise to prescribed distributions of their means; see also \citet{CifReg(93)}.

By comparison, the distributional analysis of nonlinear functionals has received considerably less attention, reflecting the greater complexity of the problem. For the variance of a Dirichlet process, exact distributional results were obtained by \citet{CifMel(00)}. They derived its moments, a Laplace-transform representation, and its density through an indirect argument based on a relation between the moments of the variance and those of the squared difference of two independent Dirichlet means. This relation leads to an integral transform involving a confluent hypergeometric function, whose analytic inversion requires a delicate analytic argument; see also \citet[Section 8.3]{LijReg(04)}. The same moment relation was used by \citet{Epi(06)} to characterize the Dirichlet variance through a stochastic equation. Thus, while these contributions provide an exact distributional description of the Dirichlet variance in terms of integral transforms and their inversion, they do not yield a direct transform identity for the variance functional itself, analogous to the Cifarelli--Regazzini identity for the mean. To the best of our knowledge, no such identity is currently available for the Dirichlet process, let alone for the Pitman--Yor process.

To fill this gap, the present paper establishes Cifarelli--Regazzini-type identities providing the generalized Cauchy--Stieltjes transforms of the Pitman--Yor and Dirichlet variances and then analytically inverts these transforms, obtaining representations for their distribution functions and densities.

\subsection{Background}

Let $\ptilde_{\theta}$ be a Dirichlet process on $\R$ with parameter $\theta>0$ and base probability measure $\G$. For any measurable function $\varphi:\R\to [0,+\infty)$, define the linear functional $\L_{\varphi}:=\int_{\R}\varphi(x)\,\ptilde_{\theta}(\ud x)$. Assuming that $\int_{\R}\log(1+\varphi(x))\,\G(\ud x)<+\infty$, \citet{CifReg(79b)} established the identity
\begin{equation}\label{eq:cf_indentity}
\EE_{\theta}\left[\left(\frac{1}{z+\L_{\varphi}}\right)^{\theta}\right]=\exp\left\{-\theta\int_{\R}\log\{z+\varphi(x)\}\,\G(\ud x)\right\},
\end{equation}
for every $z\in\C\setminus(-\infty,0]$, where the logarithm is understood in its principal determination. The importance of the identity \eqref{eq:cf_indentity}, known as the Cifarelli--Regazzini identity, goes beyond the explicit evaluation of a transform of $\L_{\varphi}$. \citet{CifReg(90)} studied the analytic inversion of the generalized Cauchy--Stieltjes transform \eqref{eq:cf_indentity}, obtaining explicit expressions for the distribution function and density of the Dirichlet mean. Their work thereby established an effective transform-and-inversion strategy: first identify a tractable transform of the functional of interest, then analyze its boundary behavior, and finally recover its distribution through a suitable analytic inversion.

The extension of \eqref{eq:cf_indentity} to the Pitman--Yor process first appeared in \citet{Tsi(99)}; see also \citet{VerYorTsi(04)}. Let $\ptilde_{\alpha,\theta}$ be a Pitman--Yor process on $\R$ with parameters $\alpha\in(0,1)$ and $\theta>0$, and base probability measure $\G$. For any measurable function $\varphi:\R\to[0,+\infty)$, define $\L_{\varphi}:=\int_{\R}\varphi(x)\,\ptilde_{\alpha,\theta}(\ud x)$. Assuming that $\int_{\R}[\varphi(x)]^{\alpha}\,\G(\ud x)<+\infty$, its generalized Cauchy--Stieltjes transform is
\begin{equation}\label{eq:CR-introduction}
\EE_{(\alpha,\theta)}\left[\left(\frac{1}{z+\L_{\varphi}}\right)^{\theta}\right]=\left\{\int_{\R}[z+\varphi(x)]^{\alpha}\,\G(\ud x)\right\}^{-\theta/\alpha},
\end{equation}
for every $z\in\C\setminus(-\infty,0]$, where complex powers are understood in their principal determination. In the sequel, we refer to the identity \eqref{eq:CR-introduction} as the Tsilevich identity. If, in addition, $\int_{\R}\varphi(x)^{\alpha_0}\,\G(\ud x)<+\infty$ for some $\alpha_0>0$, the right-hand side of \eqref{eq:CR-introduction} converges, as $\alpha\downarrow0$, to the right-hand side of the Cifarelli--Regazzini identity \eqref{eq:cf_indentity}. \citet{Jam(08)} studied the analytic inversion of the Tsilevich identity, obtaining explicit formulas for the distribution functions and densities of linear functionals of the Pitman--Yor process, and also established several useful distributional identities.

Despite its richness, the distributional theory of functionals of random probability measures remains largely a theory of linear functionals. The reason is structural: the image-measure argument that reduces a linear functional $\L_{\varphi}$ to a mean is no longer available for nonlinear functionals. This obstruction already appears for the variance of the Pitman--Yor process $\ptilde_{\alpha,\theta}$. If
\begin{displaymath}
\M_1:=\int_{\R}x\,\ptilde_{\alpha,\theta}(\ud x),\qquad\M_2:=\int_{\R}x^2\,\ptilde_{\alpha,\theta}(\ud x),
\end{displaymath}
then the Pitman--Yor variance is
\begin{equation}\label{eq:variance-introduction}
\V=\M_2-\M_1^2=\frac{1}{2}\int_{\R}\int_{\R}(x-y)^2\,\ptilde_{\alpha,\theta}(\ud x)\,\ptilde_{\alpha,\theta}(\ud y).
\end{equation}
The latter representation makes explicit the quadratic dependence of $\V$ on $\ptilde_{\alpha,\theta}$. Although $\M_1$ and $\M_2$ are individually linear functionals, their combination in \eqref{eq:variance-introduction} cannot be reduced to a mean.

\subsection{Our contributions}

In this paper, we develop a transform-and-inversion strategy for the variance $\V$ of the Pitman--Yor process, in the spirit of the original body of work by \citet{CifReg(79a),CifReg(79b),CifReg(90)}. For $\alpha\in(0,1)$, a first route to a Cifarelli--Regazzini-type identity for the Pitman--Yor variance is provided by a bivariate Tsilevich identity for two linear functionals, which was stated in \citet{KeTsi(04)} as a counterpart of their bivariate Cifarelli--Regazzini identity. We provide a probabilistic-combinatorial proof of the bivariate Tsilevich identity which, besides being of independent interest, yields a rigorous derivation of the Cifarelli--Regazzini-type identity for the variance $\V$. In particular, applying this identity to the functions $x\mapsto x$ and $x\mapsto x^2$ makes it possible to handle the joint powers of the first two moments arising in the expansion of the variance $\V$ displayed in \eqref{eq:variance-introduction}. This leads to a representation of the generalized Cauchy--Stieltjes transform of $\V$, which can then be rewritten as a Cifarelli--Regazzini-type identity. Along the same lines, we obtain the corresponding Cifarelli--Regazzini-type identity for the Dirichlet variance.  The following theorem provides one of the main results of the paper, namely, Cifarelli--Regazzini-type identities for the Pitman--Yor and Dirichlet variances. For convenience, in our treatment of the variance, we restrict attention to Pitman--Yor and Dirichlet processes whose base probability measure $\G$ is supported on $[0,+\infty)$. Throughout, $B(a,b)$ denotes Euler's Beta function.

\begin{thm}\label{thm:CR_variance}
Let $\alpha\in[0,1)$ and $\theta>1/2$, and let $\ptilde_{\alpha,\theta}$ be a Pitman--Yor process on $[0,+\infty)$ with base probability measure $\G$, where $\alpha=0$ corresponds to the Dirichlet process. Further, let $\V$ be the variance of $\ptilde_{\alpha,\theta}$ defined in \eqref{eq:variance-introduction}, and set $p:=\theta-\frac12$. Then the following identities for $\V$ hold.
\begin{itemize}
\item[i)] If $\alpha\in(0,1)$ and
$\int_0^{+\infty}x^{2\alpha}\,\G(\ud x)<+\infty$, then, for every $z\in\C\setminus(-\infty,0]$,
\begin{equation}\label{eq:variance-CR-introduction}
\mathcal T_p(z):=\EE_{(\alpha,\theta)}\left[\left(\frac{1}{z+\V}\right)^p\right]=\frac{1}{B(1/2,p)}\int_{\R}\left\{\int_0^{+\infty}\bigl[(x-s)^2+z\bigr]^\alpha\,\G(\ud x)\right\}^{-\theta/\alpha}\ud s.
\end{equation}
\item[ii)] If $\alpha=0$ and
$\int_0^{+\infty}\log(1+x^2)\,\G(\ud x)<+\infty$, then, for every $z\in\C\setminus(-\infty,0]$,
\begin{equation}\label{eq:variance-CR-introduction-DP}
\mathcal T_p(z):=\EE_{(0,\theta)}\left[\left(\frac{1}{z+\V}\right)^p\right]=\frac{1}{B(1/2,p)}\int_{\R}\exp\left\{-\theta\int_0^{+\infty}\operatorname{Log}\bigl((x-s)^2+z\bigr)\,\G(\ud x)\right\}\ud s.
\end{equation}
\end{itemize}
In both \eqref{eq:variance-CR-introduction} and \eqref{eq:variance-CR-introduction-DP} complex powers and logarithms are understood in their principal determination.
\end{thm}

For $\theta>1/2$, Theorem~\ref{thm:CR_variance} provides Cifarelli--Regazzini-type identities for the Pitman--Yor and Dirichlet variances. In particular, the generalized Cauchy--Stieltjes transforms \eqref{eq:variance-CR-introduction}--\eqref{eq:variance-CR-introduction-DP} are expressed directly in terms of the base probability measure $\G$. For $\alpha\in(0,1)$, the transform identity \eqref{eq:variance-CR-introduction} also yields a distributional representation of the Pitman--Yor variance as a variance-related functional of a Dirichlet process whose base probability measure is the joint distribution of the mean and variance of a Pitman--Yor process with parameter $\theta=0$. This representation mirrors the classical within--between decomposition of variance and provides a quadratic counterpart to the distributional identity for Pitman--Yor means in \citet[Theorem~2.1]{Jam(08)}.

For $\theta>1/2$, we analytically invert these generalized Cauchy--Stieltjes transforms under general conditions on the base probability measure $\G$, obtaining explicit representations for the distribution functions of the variances; for $\theta>3/2$, we also derive explicit representations for their densities. For $\alpha\in(0,1)$, these representations are new, whereas for $\alpha=0$ they provide more explicit formulas than the representations of the distribution function and density of the Dirichlet variance obtained by \citet{CifMel(00)}. For the uniform base probability measure $\G$ on $[0,1]$, we show that the general formulas simplify further and are directly amenable to numerical evaluation. This gain in explicitness suggests that the generalized Cauchy--Stieltjes transform provides a natural starting point for the distributional analysis of the Pitman--Yor and Dirichlet variances, in much the same way as for the corresponding means. Beyond the functional considered here, our results show that the direct transform-and-inversion strategy of \citet{CifReg(90)} is not intrinsically restricted to linear functionals. The variance is arguably the simplest functional with a genuinely quadratic dependence on the underlying random probability measure, and the approach developed here suggests a route towards the distributional analysis of higher-order moments and, more generally, polynomial functionals of the Pitman--Yor process.

\subsection{Related work}

Beyond the distributional theory of Dirichlet means reviewed above, a broader literature has investigated structural and analytic properties of functionals of Dirichlet processes. Hoeffding-type decompositions and orthogonal representations for functionals of Dirichlet processes were studied by \citet{Pec(04),ElDPec(08)}, while \citet{Pec(08)} developed multiple-integral representations for such functionals. \citet{DelSch(19)} investigated characteristic functionals of Dirichlet processes, and \citet{FliTor(23)} established an integration-by-parts formula for functionals of the Dirichlet process. Related infinite-dimensional analytic developments include the probability-measure-valued diffusion models with Dirichlet invariant measure studied by \citet{DelSch(22)}. Most recently, \citet{LasPic(26)} developed a stochastic calculus for the Dirichlet process, including chaos expansions, integration-by-parts formulas, and a Malliavin-type calculus. These contributions provide a complementary perspective on the analysis of functionals of Dirichlet processes, whereas the present paper focuses on the exact distributional analysis of the variance through generalized Cauchy--Stieltjes transforms and analytic inversion.

\subsection{Organization of the paper}

The paper is structured as follows. Section~\ref{sec2} introduces the Pitman--Yor process, recalls its combinatorial sampling structure through the Ewens--Pitman sampling model, and provides a probabilistic-combinatorial proof of the Tsilevich identity. Section~\ref{sec4} is devoted to the distributional analysis of the Pitman--Yor and Dirichlet variances: it gives a probabilistic-combinatorial proof of the bivariate Tsilevich identity, derives Cifarelli--Regazzini-type identities for the variance, investigates a distributional relationship between Pitman--Yor and Dirichlet variances, and obtains explicit representations for their distribution functions and densities through analytic inversion. Section~\ref{sec_num} includes numerical illustrations of the resulting formulas for the Pitman--Yor and Dirichlet variances in the case of a uniform base probability measure on $[0,1]$. Section~\ref{sec5} presents a generalized Cauchy--Stieltjes transform identity for the covariance functional of the Pitman--Yor process and discusses open problems concerning nonlinear functionals, with emphasis on the Gini mean difference. Appendix~\ref{app:linear_inversion} revisits the analytic inversion of the Tsilevich identity for linear functionals of the Pitman--Yor process, providing a unified representation of their densities for all $\theta>0$, whereas Appendix~\ref{app:covariance-transform} contains the proof of the covariance transform identity.


\section{Pitman--Yor processes and the Tsilevich identity}\label{sec2}

This section collects the main definitions and results that will be used in the distributional analysis of the Pitman--Yor variance. We first recall the Pitman--Yor process and its combinatorial sampling structure in terms of the Ewens--Pitman sampling formula. We then review a compound Poisson representation of the Ewens--Pitman sampling formula. Finally, we use these ingredients to give a probabilistic-combinatorial proof of the Tsilevich identity \eqref{eq:CR-introduction}. Although the Tsilevich identity is known in the literature \citep{Tsi(99),VerYorTsi(04)}, our derivation illustrates an alternative approach that will subsequently be used to establish its bivariate counterpart.

\subsection{The Pitman--Yor process and its combinatorial sampling structure}

There are several equivalent constructions of the Pitman--Yor process \citep{PerPitYor(92),Pit(95),PitYor(97)}. Among them, the stick-breaking construction \citep{PerPitYor(92),Pit(95)} provides a particularly simple and explicit description of the process. Let $\alpha \in [0,1)$, $\theta > -\alpha$, and let $\G$ be a probability measure on $(\R,\BR)$. Consider a sequence $\{V_n\}_{n\ge1}$ of independent random variables such that $V_n\sim\mathrm{Beta}(1-\alpha,\theta+n\alpha)$. Set $\omega_1:=V_1$ and, for $n\ge2$, define $\omega_n:=V_n\prod_{i=1}^{n-1}(1-V_i)$. Next, let $\{\xi_n\}_{n\ge1}$ be a sequence of independent and identically distributed random variables with common distribution $\G$, independent of $\{V_n\}_{n\ge1}$. The Pitman--Yor process with parameters $\alpha\in[0,1)$, $\theta>-\alpha$ and base probability measure $\G$ is then defined as
\begin{equation}\label{stick}
\ptilde_{\alpha,\theta}:=\sum_{n=1}^{+\infty}\omega_n\delta_{\xi_n},
\end{equation}
where $\delta_x$ denotes the Dirac probability measure centered at $x\in\R$. When $\alpha=0$, the Pitman--Yor process reduces to the Dirichlet process \citep{Fer(73)}. We refer to \citet[Chapter 4]{Pit(06)} for a detailed account of constructions of the Pitman--Yor process and related classes of random probability measures.

Let $\{X_n\}_{n\geq1}$ be a sequence of exchangeable random variables directed by $\ptilde_{\alpha,\theta}$, that is, $\PP_{(\alpha,\theta)}[X_1\in A_1,\ldots,X_n\in A_n]=\EE_{(\alpha,\theta)}\left[\prod_{i=1}^n\ptilde_{\alpha,\theta}(A_i)\right]$ for every $n\in\N$ and $A_1,\ldots,A_n\in\BR$. Since the Pitman--Yor process $\ptilde_{\alpha,\theta}$ is almost surely discrete, the sample $(X_1,\ldots,X_n)$ induces a random partition of the set $\{1,\ldots,n\}$ through the equivalence relation $i\sim j\leftrightarrow X_i=X_j$ almost surely; see \citet[Chapter 2]{Pit(06)}. For any $n\in\N$, let $\mathcal P^\downarrow(n):=\bigcup_{k=1}^n\mathcal P^\downarrow(n;k)$, where, for $k\in\{1,\ldots,n\}$,
\begin{displaymath}
\mathcal P^\downarrow(n;k):=\left\{(\nu_1,\ldots,\nu_k)\in\N^k:\nu_1\geq\nu_2\geq\cdots\geq\nu_k,\quad\sum_{j=1}^k\nu_j=n\right\}.
\end{displaymath}
We denote by $K_n$ the number of blocks of the random partition of $\{1,\ldots,n\}$ induced by the sample $(X_1,\ldots,X_n)$, by $(N_1,\ldots,N_{K_n})\in\mathcal P^\downarrow(n)$ the corresponding block sizes arranged in non-increasing order, and by $M_{j,n}:=\#\left\{i\in\{1,\ldots,K_n\}:N_i=j\right\}$, for $j=1,\ldots,n$, the number of blocks of size $j$. Accordingly, $\sum_{j=1}^n jM_{j,n}=n$ and $\sum_{j=1}^n M_{j,n}=K_n$. This motivates the definition
\begin{displaymath}
\mathcal M(n):=\left\{(m_1,\ldots,m_n)\in\N_0^n:\sum_{j=1}^n jm_j=n\right\}.
\end{displaymath}

For parameters $\alpha\in(0,1)$ and $\theta>-\alpha$, define, for any $(x_1,\ldots,x_n)\in\mathcal M(n)$ and $s_n:=\sum_{i=1}^n x_i$,
\begin{equation}\label{coeffic}
\mathcal E_n^{(\alpha,\theta)}(x_1,\ldots,x_n):=n!\,\frac{(\theta/\alpha+1)_{s_n-1}}{\alpha(\theta+1)_{n-1}}\prod_{i=1}^n\frac{1}{x_i!}
\left(\frac{\alpha(1-\alpha)_{i-1}}{i!}\right)^{x_i},
\end{equation}
where $(z)_n:=\prod_{i=0}^{n-1}(z+i)$ denotes the rising factorial, with the proviso that $(z)_0:=1$. For any $\mathbf n:=(n_1,\ldots,n_k)\in\mathcal P^\downarrow(n;k)$, define $m_j(\mathbf n):=\#\left\{i\in\{1,\ldots,k\}:n_i=j\right\}$, for $j=1,\ldots,n$, and set
\begin{displaymath}
\overline{\mathcal E}_n^{(\alpha,\theta)}(n_1,\ldots,n_k):=\mathcal E_n^{(\alpha,\theta)}\bigl(m_1(\mathbf n),\ldots,m_n(\mathbf n)\bigr).
\end{displaymath}
Now, let $\Pi_n(n_1,\ldots,n_k)$ denote the collection of all unordered set partitions $\{I_1,\ldots,I_k\}$ of $\{1,\ldots,n\}$ having block sizes $\#I_j=n_j$, for $j=1,\ldots,k$. Accordingly, the cardinality of this set is
\begin{displaymath}
\#\Pi_n(n_1,\ldots,n_k)=\binom{n}{n_1\cdots n_k}\frac{1}{m_1(\mathbf n)!\cdots m_n(\mathbf n)!}.
\end{displaymath}
Under this notation, the joint distribution $\mu_n$ of the random sample $(X_1,\ldots,X_n)$ admits the representation
\begin{align}
\mu_n(A_1\times\cdots\times A_n)\label{sang}
&:=\PP_{(\alpha,\theta)}[X_1\in A_1,\ldots,X_n\in A_n]\\
&\notag=\sum_{k=1}^n\sum_{\mathbf n\in\mathcal P^\downarrow(n;k)}\overline{\mathcal E}_n^{(\alpha,\theta)}(n_1,\ldots,n_k)\frac{m_1(\mathbf n)!\cdots m_n(\mathbf n)!}{\binom{n}{n_1\cdots n_k}}\\
&\notag\qquad\times\sum_{\{I_1,\ldots,I_k\}\in\Pi_n(n_1,\ldots,n_k)}\prod_{j=1}^k\G\left(\bigcap_{i\in I_j}A_i\right),
\end{align}
for every $n\in\N$ and $A_1,\ldots,A_n\in\BR$. We refer to \citet[Section 3]{San(06)} for details on the distribution \eqref{sang}.

When the base probability measure $\G$ is atomless, the values associated with distinct blocks in the representation \eqref{sang} of the joint distribution of $(X_{1},\ldots,X_{n})$ are almost surely distinct. Consequently, the corresponding partition coincides with the random partition induced by the observed ties among $(X_1,\ldots,X_n)$. It follows that $\mathcal E_n^{(\alpha,\theta)}$ is precisely the distribution of $(M_{1,n},\ldots,M_{n,n})$, that is,
\begin{displaymath}
\PP_{(\alpha,\theta)}[M_{1,n}=x_1,\ldots,M_{n,n}=x_n]=\mathcal E_n^{(\alpha,\theta)}(x_1,\ldots,x_n),\qquad(x_1,\ldots,x_n)\in\mathcal M(n).
\end{displaymath}
In the atomless case, $\mathcal E_n^{(\alpha,\theta)}$ is precisely the distribution of the block-size multiplicities under the Ewens--Pitman sampling formula \citep{Ewe(72),Pit(95)}; see \citet[Chapter 2]{Pit(06)} and references therein. Throughout the paper, we use the term Ewens--Pitman sampling formula to refer to the coefficients $\mathcal E_n^{(\alpha,\theta)}$ in \eqref{coeffic}, irrespective of whether the base probability measure $\G$ is atomless.

\subsection{A compound Poisson representation of the Ewens--Pitman sampling formula}

We recall a compound Poisson representation of the Ewens--Pitman sampling formula $\mathcal{E}_n^{(\alpha,\theta)}$ established in \citet[Theorem 2]{Dol(21)}. See also \citet[Section 2]{Dol(20)}. For $\alpha\in(0,1)$, let $f_\alpha$ denote the density of a positive $\alpha$-stable random variable, uniquely characterized by
\begin{displaymath}
\int_0^{+\infty}e^{-tz}f_\alpha(z)\,\ud z=e^{-t^\alpha},\qquad t>0.
\end{displaymath}
For $\theta>-\alpha$, consider the density
\begin{displaymath}
f_{\alpha,\theta}(z)=\frac{\Gamma(\theta+1)}{\alpha\Gamma(\theta/\alpha+1)}z^{(\theta-1)/\alpha-1}f_\alpha(z^{-1/\alpha}),\qquad z>0,
\end{displaymath}
corresponding to the distribution of a scaled Mittag--Leffler random variable, denoted by $S_{\alpha,\theta}$ \citep[Chapter 0]{Pit(06)}. Let $G_{\theta+n,1}$ be a Gamma random variable with shape parameter $\theta+n$ and rate parameter $1$, independent of $S_{\alpha,\theta}$. From \citet[Equation 12]{Dol(20)}, $G_{\theta+n,1}^{\alpha}S_{\alpha,\theta}$ has density
\begin{equation}\label{eq:density_fn}
f_{n,\alpha,\theta}(z)=\frac{z^{\theta/\alpha-1}e^{-z}}{\alpha\Gamma(\theta/\alpha+1)(\theta+1)_{n-1}}\sum_{k=1}^n\mathscr{C}(n,k,\alpha)z^k,\qquad z>0,
\end{equation}
where
\begin{displaymath}
\mathscr{C}(n,k,\alpha):=\frac{1}{k!}\sum_{i=1}^k\binom{k}{i}(-1)^i(-i\alpha)_n
\end{displaymath}
is the $(n,k)$-th generalized factorial coefficient; see \citet[Chapter 2]{Cha(05)} for details. With this notation, the Ewens--Pitman sampling formula $\mathcal{E}_n^{(\alpha,\theta)}$ admits the mixture representation
\begin{equation}\label{eq:DF}
\mathcal{E}_n^{(\alpha,\theta)}(x_1,\ldots,x_n)=\int_0^{+\infty}\mathcal{F}_n^{(z,\alpha)}(x_1,\ldots,x_n)f_{n,\alpha,\theta}(z)\,\ud z,
\end{equation}
where, for $z>0$ and $(x_1,\ldots,x_n)\in\mathcal{M}(n)$,
\begin{equation}\label{eq:Charalambides_fn}
\mathcal{F}_n^{(z,\alpha)}(x_1,\ldots,x_n):=\frac{n!z^{s_n}}{\sum_{k=1}^n\mathscr{C}(n,k,\alpha)z^k}\prod_{i=1}^n\frac{1}{x_i!}\left(\frac{\alpha(1-\alpha)_{i-1}}{i!}\right)^{x_i}.
\end{equation}
In particular, $\mathcal{F}_n^{(z,\alpha)}$ is the partition probability function associated with the distribution of a compound Poisson random partition of negative binomial type, as defined in \citet[Section 3]{Cha(07)}.

\subsection{An alternative proof of the Tsilevich identity}

We provide a probabilistic-combinatorial proof of the Tsilevich identity \eqref{eq:CR-introduction} based on the distribution \eqref{sang} and the compound Poisson representation of the Ewens--Pitman sampling formula in \eqref{eq:DF}. By the image-measure property of the Pitman--Yor process, it suffices to establish the Tsilevich identity for the mean functional, replacing the base probability measure $\G$ with $\G\circ\varphi^{-1}$. We first prove the identity for a base probability measure supported on $[0,1]$, then extend it to general non-negative linear functionals, and finally establish its validity on the complex domain $\C\setminus(-\infty,0]$.

\paragraph{\normalfont\itshape Mean functional and bounded support.}
 Fix $\alpha\in(0,1)$ and $\theta>0$, and let $\G$ be supported on $[0,1]$. Let $\L:=\int_0^1 x\,\ptilde_{\alpha,\theta}(\ud x)$ be the Pitman--Yor mean. The moments of $\L$ can be expressed through the sampling representation \eqref{sang} of the Pitman--Yor process; see \citet{Yam(84)} for the Dirichlet case. Applying \eqref{sang} and grouping the terms according to the block-size multiplicities yields
\begin{equation}\label{eq:Yamato}
\EE_{(\alpha,\theta)}[\L^n]=\sum_{(x_1,\ldots,x_n)\in\mathcal M(n)}\mathcal E_n^{(\alpha,\theta)}(x_1,\ldots,x_n)\prod_{i=1}^n[\m_i(\G)]^{x_i},
\end{equation}
where $\m_i(\G):=\int_0^1x^i\,\G(\ud x)$ denotes the $i$-th moment of $\G$. Combining \eqref{eq:Yamato} with the compound Poisson representation of the Ewens--Pitman sampling formula \eqref{eq:DF}, together with \eqref{eq:density_fn} and \eqref{eq:Charalambides_fn}, gives
\begin{align}\label{mom_for}
\EE_{(\alpha,\theta)}[\L^n]&=\int_0^{+\infty}\left\{\sum_{(x_1,\ldots,x_n)\in\mathcal M(n)}\mathcal F_n^{(z,\alpha)}(x_1,\ldots,x_n)\prod_{i=1}^n[\m_i(\G)]^{x_i}
\right\}f_{n,\alpha,\theta}(z)\,\ud z\\
&\notag=\frac{n!}{(\theta)_n}\int_0^{+\infty}\left\{\sum_{(x_1,\ldots,x_n)\in\mathcal M(n)}\prod_{i=1}^n\frac{1}{x_i!}\left[\frac{z\alpha(1-\alpha)_{i-1}\m_i(\G)}{i!}\right]^{x_i}\right\}\frac{1}{\Gamma(\theta/\alpha)}z^{\theta/\alpha-1}e^{-z}\,\ud z.
\end{align}
To evaluate the inner sum, recall the exponential Fa\`a di Bruno formula \citep[Chapter 1]{Pit(06)}
\begin{displaymath}
\frac{\ud^n}{\ud t^n}e^{\phi(t)}=n!e^{\phi(t)}\sum_{(x_1,\ldots,x_n)\in\mathcal M(n)}\prod_{i=1}^n\frac{1}{x_i!}\left(\frac{\phi^{(i)}(t)}{i!}\right)^{x_i}.
\end{displaymath}
In particular, if $\phi(t)=\sum_{k=1}^{+\infty}a_kt^k$ has positive radius of convergence, then
\begin{displaymath}
\left.\frac{\ud^n}{\ud t^n}\exp\left\{\sum_{k=1}^{+\infty}a_kt^k\right\}\right|_{t=0}=n!\sum_{(x_1,\ldots,x_n)\in\mathcal M(n)}\prod_{i=1}^n\frac{a_i^{x_i}}{x_i!}.
\end{displaymath}
Using $\sum_{k=1}^{+\infty}\frac{\Gamma(k-\alpha)}{k!}u^k=-\Gamma(-\alpha)\bigl[1-(1-u)^\alpha\bigr]$, for $|u|<1$, together with $\Gamma(1-\alpha)=-\alpha\Gamma(-\alpha)$, one obtains
\begin{align*}
&n!\sum_{(x_1,\ldots,x_n)\in\mathcal M(n)}\prod_{i=1}^n\frac{1}{x_i!}\left[\frac{z\alpha(1-\alpha)_{i-1}\m_i(\G)}{i!}\right]^{x_i}\\
&\quad=\left.\frac{\ud^n}{\ud t^n}\exp\left\{z\alpha\sum_{k=1}^{+\infty}t^k\frac{\Gamma(k-\alpha)}{\Gamma(1-\alpha)}\frac{1}{k!}\int_0^1y^k\,\G(\ud y)\right\}\right|_{t=0}\\
&\quad=\left.\frac{\ud^n}{\ud t^n}\exp\left\{\frac{z\alpha}{\Gamma(1-\alpha)}\int_0^1\left(\sum_{k=1}^{+\infty}\frac{\Gamma(k-\alpha)}{k!}(ty)^k\right)\G(\ud y)\right\}\right|_{t=0}\\
&\quad=\left.\frac{\ud^n}{\ud t^n}\exp\left\{-\frac{z\alpha\Gamma(-\alpha)}{\Gamma(1-\alpha)}\int_0^1[1-(1-ty)^\alpha]\,\G(\ud y)\right\}\right|_{t=0}\\
&\quad=\left.\frac{\ud^n}{\ud t^n}\exp\left\{z\int_0^1[1-(1-ty)^\alpha]\,\G(\ud y)\right\}\right|_{t=0}.
\end{align*}
Substituting this last expression into the moment formula \eqref{mom_for} and interchanging differentiation and integration, which is justified in a neighborhood of $t=0$ by the bounded support of $\G$, yields
\begin{align*}
\EE_{(\alpha,\theta)}[\L^n]&=\frac{1}{(\theta)_n}\left.\frac{\ud^n}{\ud t^n}\left[\int_0^{+\infty}\exp\left\{z\int_0^1[1-(1-ty)^\alpha]\,\G(\ud y)\right\}\frac{z^{\theta/\alpha-1}e^{-z}}{\Gamma(\theta/\alpha)}\,\ud z\right]\right|_{t=0}\\
&=\frac{1}{(\theta)_n}\left.\frac{\ud^n}{\ud t^n}\left[\left(1-\int_0^1[1-(1-ty)^\alpha]\,\G(\ud y)\right)^{-\theta/\alpha}\right]\right|_{t=0}\\
&=\frac{1}{(\theta)_n}\left.\frac{\ud^n}{\ud t^n}\left[\left(\int_0^1(1-ty)^\alpha\,\G(\ud y)\right)^{-\theta/\alpha}\right]\right|_{t=0}.
\end{align*}
Since $0\leq\L\leq1$ almost surely, for $s\in(-1,1)$ the negative binomial expansion is absolutely convergent. Therefore,
\begin{align*}
\EE_{(\alpha,\theta)}\left[\left(\frac{1}{1+s\L}\right)^\theta\right]&=\sum_{n=0}^{+\infty}\frac{(\theta)_n}{n!}(-s)^n\EE_{(\alpha,\theta)}[\L^n]\\
&=\sum_{n=0}^{+\infty}\frac{(-s)^n}{n!}\left.\frac{\ud^n}{\ud t^n}\left[\left(\int_0^1(1-ty)^\alpha\,\G(\ud y)\right)^{-\theta/\alpha}\right]\right|_{t=0}\\
&=\left(\int_0^1(1+sy)^\alpha\,\G(\ud y)\right)^{-\theta/\alpha},
\end{align*}
for every $s\in(-1,1)$. Both sides define holomorphic functions of $s$ on the half-plane $\Re(s)>-1$. Indeed, since $0\leq\L\leq1$, for $\Re(s)>-1$ one has $\Re(1+s\L)>0$ whenever $\L>0$, while $1+s\L=1$ if $\L=0$. Similarly, for every $x\in[0,1]$, $\Re(1+sx)>0$ whenever $x>0$, while $1+sx=1$ if $x=0$. Hence neither $1+s\L$ nor $1+sx$ meets the nonpositive real axis, and the corresponding principal powers are holomorphic in $s$. Moreover, since $-\frac{\pi}{2}<\arg(1+sx)<\frac{\pi}{2}$, we have $-\frac{\alpha\pi}{2}<\arg\{(1+sx)^\alpha\}<\frac{\alpha\pi}{2}$, 
and therefore $\Re\{(1+sx)^\alpha\}>0$. Consequently, $\Re\{\int_0^1(1+sx)^\alpha\,\G(\ud x)\}>0$, so that the integral does not vanish and remains in the right half-plane. Thus, its principal power $-\theta/\alpha$ is holomorphic as well. Since the two sides coincide for $s\in(0,1)$, the identity theorem yields the desired identity throughout the half-plane $\Re(s)>-1$, and in particular for every $s>0$.

\paragraph{\normalfont\itshape Extension to general non-negative linear functionals.}
By a deterministic rescaling, the identity then holds for a Pitman--Yor mean associated with any probability measure having bounded support in $[0,+\infty)$. The general case follows by a standard image-measure and truncation argument. Indeed, for a non-negative measurable function $\varphi$, the linear functional $\L_\varphi=\int_{\R} \varphi(x)\,\ptilde_{\alpha,\theta}(\ud x)$ has the same distribution as the Pitman--Yor mean associated with the image probability measure $\G_\varphi:=\G\circ\varphi^{-1}$ on $[0,+\infty)$.
For $M>0$, let $\varphi_M:=\varphi\wedge M$. For every $z>0$, applying the bounded-support identity to $\varphi_M$ gives
\begin{displaymath}
\EE_{(\alpha,\theta)}\left[(z+\L_{\varphi_M})^{-\theta}\right]=\left\{\int_{\R}[z+\varphi_M(x)]^\alpha\,\G(\ud x)\right\}^{-\theta/\alpha}.
\end{displaymath}
Letting $M\to+\infty$, we have $\L_{\varphi_M}\uparrow\L_\varphi$ almost surely, where at this stage $\L_\varphi$ is allowed to take the value $+\infty$. With the convention $(z+\infty)^{-\theta}:=0$, it follows that $(z+\L_{\varphi_M})^{-\theta}\rightarrow(z+\L_\varphi)^{-\theta}$ almost surely. Since $0\leq(z+\L_{\varphi_M})^{-\theta}\leq z^{-\theta}$, the dominated convergence theorem applies to the left-hand side, while the monotone convergence theorem applied to the integral on the right-hand side gives $\int_{\R}[z+\varphi_M(x)]^\alpha\,\G(\ud x)\rightarrow
\int_{\R}[z+\varphi(x)]^\alpha\,\G(\ud x)$. Consequently,
\begin{displaymath}
\EE_{(\alpha,\theta)}\left[(z+\L_\varphi)^{-\theta}\right]=\left\{\int_{\R}[z+\varphi(x)]^\alpha\,\G(\ud x)\right\}^{-\theta/\alpha},
\end{displaymath}
for every $z>0$, provided that $\int_{\R}\varphi(x)^\alpha\,\G(\ud x)<+\infty$. The same argument also shows that this integrability condition ensures that $\L_\varphi<+\infty$ almost surely. Indeed, initially allowing $\L_\varphi$ to take the value $+\infty$, multiplication of the preceding identity by $z^\theta$ gives
\begin{displaymath}
\EE_{(\alpha,\theta)}\left[\left(1+\frac{\L_\varphi}{z}\right)^{-\theta}\right]=\left\{\int_{\R}\left(1+\frac{\varphi(x)}{z}\right)^\alpha\,\G(\ud x)
\right\}^{-\theta/\alpha}.
\end{displaymath}
As $z\to+\infty$, the left-hand side converges to $\PP_{(\alpha,\theta)}(\L_\varphi<+\infty)$, whereas the right-hand side converges to $1$ by dominated convergence and the inequality
$(1+y)^\alpha\leq1+y^\alpha$ for $y\geq0$. Hence $\PP_{(\alpha,\theta)}(\L_\varphi<+\infty)=1$.

\paragraph{\normalfont\itshape Extension to the complex domain.}
It remains to extend the identity from the positive real axis to $\C\setminus(-\infty,0]$. Both sides of \eqref{eq:CR-introduction} are holomorphic on $\C\setminus(-\infty,0]$, with all complex powers taken in their principal determination. For the left-hand side, holomorphy follows by differentiation under the expectation, since $\L_\varphi<+\infty$ almost surely and, on every compact subset $K$ of $\C\setminus(-\infty,0]$, one has $\inf_{z\in K,\,x\geq0}|z+x|>0$. 
For the right-hand side, set $I_{\alpha,\varphi}(z):=\int_{\R}[z+\varphi(x)]^\alpha\,\G(\ud x)$. The integrability assumption $\int_{\R}[\varphi(x)]^\alpha\,\G(\ud x)<+\infty$ implies that $I_{\alpha,\varphi}$ is holomorphic on $\C\setminus(-\infty,0]$. Moreover, if $\Im z>0$, then $\Im\{[z+\varphi(x)]^\alpha\}>0$ for every $x$, and hence $\Im I_{\alpha,\varphi}(z)>0$; similarly, $\Im I_{\alpha,\varphi}(z)<0$ if $\Im z<0$, while $I_{\alpha,\varphi}(z)>0$ for $z>0$. Thus $I_{\alpha,\varphi}(z)\notin(-\infty,0]$ throughout $\C\setminus(-\infty,0]$, and its principal power $-\theta/\alpha$ is holomorphic there. Since the two sides of \eqref{eq:CR-introduction} coincide for every $z>0$, the identity theorem yields \eqref{eq:CR-introduction} for every $z\in\C\setminus(-\infty,0]$. This completes the proof of the Tsilevich identity.


\section{The Pitman--Yor and Dirichlet variances} \label{sect:Variance}\label{sec4}

This section develops the distributional analysis of the variance of the Pitman--Yor process. We first provide a probabilistic-combinatorial proof of a bivariate Tsilevich identity
\citep{KeTsi(04)}, together with an expression for the joint moments of two linear functionals. We then use this identity to prove the Cifarelli--Regazzini-type identities for the Pitman--Yor and Dirichlet variances in Theorem~\ref{thm:CR_variance}. We next establish a distributional representation of the Pitman--Yor variance in terms of a variance-related functional of a Dirichlet process. Finally, by analytic inversion of the Cifarelli--Regazzini-type identities, we obtain representations for the distribution function and density of the variance under general assumptions on the base probability measure.

\subsection{A bivariate Tsilevich identity}\label{sect:4-1}

\begin{prp}\label{prp:Uno}
Let $\alpha\in(0,1)$ and $\theta>0$, and let $\ptilde_{\alpha,\theta}$ be a Pitman--Yor process on $\R$ with base probability measure $\G$. Let $\L_1:=\int_{\R}\varphi(x)\,\ptilde_{\alpha,\theta}(\ud x)$ and $\L_2:=\int_{\R}\psi(x)\,\ptilde_{\alpha,\theta}(\ud x)$, where $\varphi,\psi:\R\to\R$ are Borel-measurable functions satisfying $|\varphi(x)|\leq1$ and $|\psi(x)|\leq1$ for $\G$-almost every $x$. Then, for all integers $n,m\geq0$,
\begin{equation}\label{eq:joint_mom_linear2}
\EE_{(\alpha,\theta)}\left[\L_1^n\L_2^m\right]=\frac{1}{(\theta)_{n+m}}\left.\frac{\partial^{n+m}}{\partial u^n\partial v^m}\left\{\int_{\R}[1-u\varphi(x)-v\psi(x)]^\alpha\,\G(\ud x)
\right\}^{-\theta/\alpha}\right|_{u=v=0}.
\end{equation}
Consequently, for all $s,t\in\C$ such that $|s|+|t|<1$,
\begin{equation}\label{eq:alpha-CR-double}
\EE_{(\alpha,\theta)}\left[\left(\frac{1}{1+s\L_1+t\L_2}\right)^\theta\right]=\left\{\int_{\R}[1+s\varphi(x)+t\psi(x)]^\alpha\,\G(\ud x)\right\}^{-\theta/\alpha}.
\end{equation}
In \eqref{eq:alpha-CR-double} complex powers are understood in their principal determination. The boundedness assumption can be dropped when $\varphi,\psi:\R\to[0,+\infty)$ are Borel-measurable and satisfy $\int_{\R}[\varphi(x)^\alpha+\psi(x)^\alpha]\,\G(\ud x)<+\infty$. In this case, \eqref{eq:alpha-CR-double} holds for every $s,t>0$.
\end{prp}

\begin{proof}
Let $n,m\in\N_0$ and set $N:=n+m$. The case $N=0$ is immediate, so assume $N\geq1$. By de Finetti's representation theorem and the representation \eqref{sang} of the joint law of the sample,
\begin{align*}
\EE_{(\alpha,\theta)}\left[\L_1^n\L_2^m\right]
&=\sum_{k=1}^{N}\sum_{\mathbf N\in\mathcal P^\downarrow(N;k)}\overline{\mathcal E}_{N}^{(\alpha,\theta)}(N_1,\ldots,N_k)\\
&\qquad\times\sum_{(\ast)}\frac{1}{\binom{N}{n}}\prod_{i=1}^k\left\{\binom{N_i}{\nu_i}\int_{\R}\varphi(x)^{\nu_i}\psi(x)^{N_i-\nu_i}\,\G(\ud x)\right\},
\end{align*}
where $\mathbf N=(N_1,\ldots,N_k)$, $\sum_{i=1}^kN_i=N$, and $(\ast)$ denotes summation over all $(\nu_1,\ldots,\nu_k)\in\N_0^k$ such that $0\leq\nu_i\leq N_i$ and $\sum_{i=1}^k\nu_i=n$. For fixed block sizes $(N_1,\ldots,N_k)$, the factor $\binom{N}{n}^{-1}\prod_{i=1}^k\binom{N_i}{\nu_i}$ is the multivariate hypergeometric probability that the $n$ indices associated with $\varphi$ are distributed among the blocks with counts $(\nu_1,\ldots,\nu_k)$. Indeed,
\begin{displaymath}
\frac{\binom{n}{\nu_1,\ldots,\nu_k}\binom{m}{N_1-\nu_1,\ldots,N_k-\nu_k}}{\binom{N}{N_1,\ldots,N_k}}=\frac{\prod_{i=1}^k\binom{N_i}{\nu_i}}{\binom{N}{n}}.
\end{displaymath}
For fixed $\mathbf N$, the inner sum can be identified by expanding $\prod_{i=1}^k\int_{\R}[\psi(x)+t\varphi(x)]^{N_i}\,\G(\ud x)$. Indeed, the coefficient of $t^n$ in this product is
\begin{displaymath}
\sum_{(\ast)}\prod_{i=1}^k\left\{\binom{N_i}{\nu_i}\int_{\R}\varphi(x)^{\nu_i}\psi(x)^{N_i-\nu_i}\,\G(\ud x)\right\}.
\end{displaymath}
Since $\binom{N}{n}^{-1}\frac{1}{n!}=\frac{m!}{N!}$, it follows that
\begin{displaymath}
\EE_{(\alpha,\theta)}\left[\L_1^n\L_2^m\right]=\frac{m!}{N!}\left.\frac{\ud^n}{\ud t^n}\left\{\sum_{k=1}^{N}\sum_{\mathbf N\in\mathcal P^\downarrow(N;k)}\overline{\mathcal E}_{N}^{(\alpha,\theta)}(\mathbf N)\right.\right.\left.\left.\prod_{i=1}^k\int_{\R}[\psi(x)+t\varphi(x)]^{N_i}\,\G(\ud x)\right\}\right|_{t=0}.
\end{displaymath}
By the same grouping-by-block-sizes argument used above in deriving \eqref{eq:Yamato} from the general sampling representation \eqref{sang}, the expression inside braces is the $N$-th moment of the linear functional associated with the function $\psi+t\varphi$. The univariate moment argument of the previous section extends directly to bounded, possibly signed, measurable integrands in a neighborhood of the origin. Applying it to the bounded function $\psi+t\varphi$ gives
\begin{align*}
&\sum_{k=1}^{N}\sum_{\mathbf N\in\mathcal P^\downarrow(N;k)}\overline{\mathcal E}_{N}^{(\alpha,\theta)}(\mathbf N)\prod_{i=1}^k\int_{\R}[\psi(x)+t\varphi(x)]^{N_i}\,\G(\ud x)\\
&\qquad=\frac{1}{(\theta)_N}\left.\frac{\partial^N}{\partial z^N}\left\{\int_{\R}[1-z\{\psi(x)+t\varphi(x)\}]^\alpha\,\G(\ud x)\right\}^{-\theta/\alpha}\right|_{z=0}.
\end{align*}
Hence
\begin{equation}
\EE_{(\alpha,\theta)}\left[\L_1^n\L_2^m\right]=\frac{m!}{N!}\frac{1}{(\theta)_N}\left.\frac{\partial^{N+n}}{\partial z^N\partial t^n}\left\{\int_{\R}[1-z\psi(x)-zt\varphi(x)]^\alpha\,\G(\ud x)
\right\}^{-\theta/\alpha}\right|_{z=t=0}.
\label{eq:joint_moment_intermediate}
\end{equation}
To identify this expression with the mixed derivative in \eqref{eq:joint_mom_linear2}, define 
\begin{displaymath}
\Phi(u,v):=\left\{\int_{\R}[1-u\varphi(x)-v\psi(x)]^\alpha\,\G(\ud x)\right\}^{-\theta/\alpha}.
\end{displaymath}
Here $(u,v)\in\C^2$ with $|u|+|v|<1$, and all complex powers are taken in their principal determination. Then the function appearing in \eqref{eq:joint_moment_intermediate} is $\Phi(zt,z)$. Differentiating $n$ times with respect to $t$ gives
\begin{displaymath}
\frac{\partial^n}{\partial t^n}\Phi(zt,z)=z^n\left[\frac{\partial^n\Phi}{\partial u^n}(u,v)\right]_{u=zt,\,v=z}.
\end{displaymath}
Therefore,
\begin{displaymath}
\left.\frac{\partial^n}{\partial t^n}\Phi(zt,z)\right|_{t=0}=z^n\frac{\partial^n\Phi}{\partial u^n}(0,z).
\end{displaymath}
By an application of the $N$-th derivative with respect to $z$ and evaluating at $z=0$, Leibniz's rule yields
\begin{displaymath}
\left.\frac{\ud^N}{\ud z^N}\left\{z^n\frac{\partial^n\Phi}{\partial u^n}(0,z)\right\}\right|_{z=0}=\frac{N!}{m!}\left.\frac{\partial^{n+m}\Phi}{\partial u^n\partial v^m}
\right|_{u=v=0}.
\end{displaymath}
Substituting this identity into \eqref{eq:joint_moment_intermediate}, and recalling that $N=n+m$, gives
\begin{displaymath}
\EE_{(\alpha,\theta)}\left[\L_1^n\L_2^m\right]=\frac{1}{(\theta)_{n+m}}\left.\frac{\partial^{n+m}}{\partial u^n\partial v^m}\left\{\int_{\R}[1-u\varphi(x)-v\psi(x)]^\alpha
\,\G(\ud x)\right\}^{-\theta/\alpha}\right|_{u=v=0},
\end{displaymath}
which proves \eqref{eq:joint_mom_linear2}.

We now derive \eqref{eq:alpha-CR-double}. Since $|\varphi|\leq1$ and $|\psi|\leq1$ hold $\G$-almost everywhere, one has $|\L_1|\leq1$ and $|\L_2|\leq1$ almost surely. Hence, if $|s|+|t|<1$, then
$|s\L_1+t\L_2|\leq |s|+|t|<1$. The negative binomial expansion is therefore absolutely convergent,
and we may write
\begin{align*}
\EE_{(\alpha,\theta)}\left[\left(\frac{1}{1+s\L_1+t\L_2}\right)^\theta\right]&=\sum_{k=0}^{+\infty}\frac{(\theta)_k}{k!}(-1)^k\EE_{(\alpha,\theta)}\left[(s\L_1+t\L_2)^k\right]\\
&=\sum_{n=0}^{+\infty}\sum_{m=0}^{+\infty}\frac{(\theta)_{n+m}}{n!\,m!}(-s)^n(-t)^m\EE_{(\alpha,\theta)}\left[\L_1^n\L_2^m\right].
\end{align*}
Moreover, $\Phi$ is holomorphic on $\{(u,v)\in\C^2:|u|+|v|<1\}$, and its Taylor expansion at $(0,0)$ converges throughout this domain. Using \eqref{eq:joint_mom_linear2}, this becomes
\begin{displaymath}
\sum_{n=0}^{+\infty}\sum_{m=0}^{+\infty}\frac{(-s)^n(-t)^m}{n!\,m!}\left.\frac{\partial^{n+m}\Phi(u,v)}{\partial u^n\partial v^m}\right|_{u=v=0}=\Phi(-s,-t),
\end{displaymath}
by the bivariate Taylor expansion of $\Phi$ around $(0,0)$. Consequently,
\begin{displaymath}
\EE_{(\alpha,\theta)}\left[\left(\frac{1}{1+s\L_1+t\L_2}\right)^\theta\right]=\left\{\int_{\R}[1+s\varphi(x)+t\psi(x)]^\alpha\,\G(\ud x)\right\}^{-\theta/\alpha},
\end{displaymath}
which proves \eqref{eq:alpha-CR-double} for $|s|+|t|<1$. Finally, suppose that $\varphi$ and $\psi$ are non-negative and $\int_{\R}\left[\varphi(x)^\alpha+\psi(x)^\alpha\right]\G(\ud x)<+\infty$. For arbitrary $s,t>0$, set $h(x):=s\varphi(x)+t\psi(x)$. Since $0<\alpha<1$, $h(x)^\alpha\leq s^\alpha\varphi(x)^\alpha+t^\alpha\psi(x)^\alpha$, and hence $\int_{\R}h(x)^\alpha\,\G(\ud x)<+\infty$. Applying the Tsilevich identity \eqref{eq:CR-introduction} with $z=1$ to the linear functional associated with $h$ yields \eqref{eq:alpha-CR-double}.
\end{proof}

\subsection{Proof of the Cifarelli--Regazzini-type identities for the variance}\label{sect:4-2}

We provide the proof of Theorem~\ref{thm:CR_variance} based on the bivariate Tsilevich identity of Proposition~\ref{prp:Uno}. Set $p:=\theta-\frac12>0$. We first derive an auxiliary generalized Cauchy--Stieltjes transform for a base probability measure $\G$ supported on $[0,1]$, then establish the transform identities for general $\G$ in the Pitman--Yor and Dirichlet cases, and finally extend both identities to the complex domain $\C\setminus(-\infty,0]$.

\paragraph{\normalfont\itshape Variance functional and bounded support.}
We first consider $\alpha\in(0,1)$ and assume that $\G$ is supported on $[0,1]$. We derive an auxiliary representation of the generalized Cauchy--Stieltjes transform of $\V$ using the joint-moment formula \eqref{eq:joint_mom_linear2}. Define $\M_1:=\int_0^1x\,\ptilde_{\alpha,\theta}(\ud x)$ and 
$\M_2:=\int_0^1x^2\,\ptilde_{\alpha,\theta}(\ud x)$, so that $\V=\M_2-\M_1^2$. For $(u,v)$ in the domain $\mathscr D_{\V}:=\{(u,v)\in\C^2:\Re(1+ux+vx^2)>0\text{ for every }x\in[0,1]\}$, set
\begin{displaymath}
\Psi(u,v):=\left\{\int_0^1[1+ux+vx^2]^\alpha\,\G(\ud x)\right\}^{-\theta/\alpha}.
\end{displaymath}
Since $\alpha\in(0,1)$, the integral inside the power has strictly positive real part. Consequently, $\Psi$ is well defined and holomorphic on $\mathscr D_{\V}$, which contains the domain $\{(u,v)\in\C^2:|u|+|v|<1\}$. In what follows, derivatives of expressions such as $\Psi(-u,-v)$ are understood as derivatives of the corresponding composite function with respect to the displayed variables. Fix $0<\tau<(\sqrt5-1)/2$, so that $\tau+\tau^2<1$. In particular, for every $\rho\in[0,1]$,
\begin{displaymath}
2\tau\sqrt{\rho(1-\rho)}+\tau^2(1-\rho)\leq\tau+\tau^2<1.
\end{displaymath}
This ensures that the arguments used in the Taylor expansions below lie in the domain $|u|+|v|<1$ of holomorphy of $\Psi$. Since $0\leq \M_1^2\leq \M_2\leq1$, the series obtained after expanding $\V^n=(\M_2-\M_1^2)^n$ is absolutely convergent for $\tau^2<1/2$; indeed, the sum of the absolute values of the terms in the binomial expansion is bounded by $(\M_2+\M_1^2)^n\leq2^n$. Combining the negative binomial series expansion, the binomial theorem, and \eqref{eq:joint_mom_linear2}, and rearranging the resulting absolutely convergent series, one obtains
\begin{align*}
&\EE_{(\alpha,\theta)}\left[\left(\frac{1}{1+\tau^2\V}\right)^p\right]\\
&\quad=\sum_{n=0}^{+\infty}\frac{(p)_n}{n!}(-\tau^2)^n\EE_{(\alpha,\theta)}[\V^n]\\
&\quad=\frac{1}{\Gamma(p)}\sum_{n=0}^{+\infty}\frac{\Gamma(n+p)}{n!}(-\tau^2)^n\sum_{k=0}^n\binom{n}{k}(-1)^k\EE_{(\alpha,\theta)}\left[\M_1^{2k}\M_2^{n-k}\right]\\
&\quad=\frac{1}{\Gamma(p)}\sum_{n=0}^{+\infty}\Gamma(n+p)(-\tau^2)^n\sum_{k=0}^n\frac{(-1)^k}{k!(n-k)!}\frac{1}{(\theta)_{n+k}}\left.\partial_u^{2k}\partial_v^{n-k}\Psi(-u,-v)
\right|_{u=v=0}.
\end{align*}
Since $\frac{1}{(\theta)_{n+k}}=\frac{\Gamma(\theta)}{\Gamma(n+k+\theta)}$, exchanging the order of summation gives
\begin{displaymath}
\EE_{(\alpha,\theta)}\left[\left(\frac{1}{1+\tau^2\V}\right)^p\right]=\frac{\Gamma(\theta)}{\Gamma(p)}\sum_{k=0}^{+\infty}\sum_{n=k}^{+\infty}\frac{\Gamma(n+p)}{\Gamma(n+k+\theta)}
\frac{(-\tau^2)^n(-1)^k}{k!(n-k)!}\left.\partial_u^{2k}\partial_v^{n-k}\Psi(-u,-v)\right|_{u=v=0}.
\end{displaymath}
Recalling that $p=\theta-\frac12$ and setting $m=n-k$, the preceding expression can be written as
\begin{align*}
&\EE_{(\alpha,\theta)}\left[\left(\frac{1}{1+\tau^2\V}\right)^p\right]\\
&\quad=\frac{\Gamma(\theta)}{\Gamma(p)}\sum_{k=0}^{+\infty}\frac{\tau^{2k}}{\Gamma(k+1/2)k!}\left.\partial_u^{2k}\left[\sum_{m=0}^{+\infty}\frac{\Gamma(m+k+\theta-1/2)\Gamma(k+1/2)}{\Gamma(m+2k+\theta)}\frac{(-\tau^2)^m}{m!}\left.\partial_v^m\Psi(-u,-v)\right|_{v=0}\right]\right|_{u=0}.
\end{align*}
Focusing on the above series inside the square brackets, notice that Euler's Beta integral gives 
\begin{displaymath}
\frac{\Gamma(m+k+\theta-1/2)\Gamma(k+1/2)}{\Gamma(m+2k+\theta)}=\int_0^1\rho^{k-1/2}(1-\rho)^{m+k+\theta-3/2}\,\ud\rho. 
\end{displaymath}
Therefore, after fixing $u$ in a sufficiently small neighborhood of zero and interchanging the sum over $m$ with the integral over $\rho$,
the resulting Taylor series converges absolutely and uniformly in $\rho\in[0,1]$, yielding
\begin{align*}
&\int_0^1\rho^{k-1/2}(1-\rho)^{k+\theta-3/2}\left[\sum_{m=0}^{+\infty}\frac{[-\tau^2(1-\rho)]^m}{m!}\left.\partial_v^m\Psi(-u,-v)\right|_{v=0}\right]\ud\rho\\
&\qquad=\int_0^1\rho^{k-1/2}(1-\rho)^{k+\theta-3/2}\Psi\bigl(-u,\tau^2(1-\rho)\bigr)\,\ud\rho.
\end{align*}
Interchanging the integral over $\rho$ with the derivative with respect to $u$ and then, again, with the sum over $k$ yields
\begin{displaymath}
\EE_{(\alpha,\theta)}\left[\left(\frac{1}{1+\tau^2\V}\right)^p\right]=\frac{\Gamma(\theta)}{\Gamma(p)}\int_0^1\rho^{-1/2}(1-\rho)^{\theta-3/2}\left[\sum_{k=0}^{+\infty}\frac{
[\rho(1-\rho)\tau^2]^k}{\Gamma(k+1/2)k!}\left.\partial_u^{2k}\Psi\bigl(-u,\tau^2(1-\rho)\bigr)\right|_{u=0}\right]\ud\rho.
\end{displaymath}
Recall that $(2k)!=\frac{4^k}{\sqrt{\pi}}\Gamma\left(k+\frac12\right)k!$ by Legendre's duplication formula, and $\sqrt{\pi}=\Gamma(1/2)$. Thus,
\begin{align}
&\EE_{(\alpha,\theta)}\left[\left(\frac{1}{1+\tau^2\V}\right)^p\right]\nonumber\\
&\quad=\frac{\Gamma(\theta)}{\Gamma(p)\Gamma(1/2)}\int_0^1\rho^{-1/2}(1-\rho)^{\theta-3/2}\left[\sum_{k=0}^{+\infty} \frac{[2\tau\sqrt{\rho(1-\rho)}]^{2k}}{(2k)!}
\left.\partial_u^{2k}\Psi\bigl(-u,\tau^2(1-\rho)\bigr)\right|_{u=0}\right]\ud\rho.
\label{eq:last_proof_variance}
\end{align}
Recognizing the even part of the Taylor expansion gives 
\begin{displaymath}
\sum_{k=0}^{+\infty}\frac{x^{2k}}{(2k)!}\left.\partial_u^{2k}\Psi(-u,v)\right|_{u=0}=\frac{\Psi(x,v)+\Psi(-x,v)}{2}
\end{displaymath}
for any $x, v \in \R$ such that $|x| + |v| < 1$. Noticing that $\frac{\Gamma(\theta)}{\Gamma(p)\Gamma(1/2)} = (B(1/2,p))^{-1}$ and setting $t=\tau^2$ yield
\begin{align}\label{eq:variance-beta-representation}
&\EE_{(\alpha,\theta)}\left[(1+t\V)^{-p}\right]\\
&\notag\quad=\frac{1}{B(1/2,p)}\\
&\notag\quad\quad\times\int_0^1\rho^{-1/2}(1-\rho)^{p-1}\frac{1}{2}\left[\Psi\left(2\sqrt{t\rho(1-\rho)},t(1-\rho)\right)\right.\left.+\Psi\left(-2\sqrt{t\rho(1-\rho)},t(1-\rho)\right)\right]\ud\rho\ .
\end{align}

The preceding argument establishes \eqref{eq:variance-beta-representation} for all sufficiently small positive $t$. To extend the identity to every $t>0$, we show that
both sides are holomorphic on the convex domain $\Omega:=\{u+iv\in\C:\ u>0,\ v^2<4(u+1)\}$, which contains the whole positive real axis. The identity theorem
for holomorphic functions then yields the desired extension. Throughout the argument, square roots and complex powers are taken on their principal branches.
For the left-hand side, set $\C_+:=\{z \in\C : \Re z>0\}$. Since $\V\geq0$ almost surely and $|(1+t\V)^{-p}|\leq1$ for $t\in\C_+$, the dominated convergence theorem provides continuity of the left-hand side 
on $\C_+$. Moreover, Fubini's and Cauchy's theorems yield $\oint_\gamma \EE_{(\alpha,\theta)}\!\left[(1+t\V)^{-p}\right]\ud t =0$ for every closed, piecewise $C^1$ contour $\gamma$ in $\C_+$.
Morera's theorem therefore shows that the left-hand side is holomorphic on $\C_+$, and hence on $\Omega$.
For the right-hand side, set $q_\pm(t,\rho,x) :=1\pm2\sqrt{t}\sqrt{\rho(1-\rho)}\,x+t(1-\rho)x^2$. 
Writing $\sqrt{t}=a+ib$, the inequality $(\Im t)^2<4(\Re t+1)$ is equivalent to $b^2<1$. Thus, for $t\in\Omega$ and $(\rho,x)\in[0,1]^2$, one gets
\begin{displaymath}
\Re q_\pm(t,\rho,x) = \bigl(\sqrt{\rho}\pm a\sqrt{1-\rho}\,x\bigr)^2 +(1-\rho)(1-b^2x^2)>0\ .
 \end{displaymath}
By continuity and compactness, for every compact $K\subset\Omega$ there exist constants $m_K,M_K>0$ such that
$\Re q_\pm(t,\rho,x)\geq m_K$ and $|q_\pm(t,\rho,x)|\leq M_K$ for any $(t,\rho,x)\in K\times[0,1]^2$.
Consequently, for each $\rho\in[0,1]$, the functions $A_\pm(t,\rho):= \int_0^1q_\pm(t,\rho,x)^\alpha\,\G(\ud x)$ are holomorphic on $\Omega$, by dominated convergence and the combination of Morera's and Fubini's theorems. Indeed, $q_\pm(\,\cdot\,,\rho,x)^\alpha$ are holomorphic and their absolute values are bounded by $M_K^\alpha$ on $K$, uniformly in $\rho,x$.
Moreover, since $0<\alpha<1$,
\begin{displaymath}
\Re A_\pm(t,\rho) \geq m_K^\alpha\cos\!\left(\frac{\pi\alpha}{2}\right) =: c_K>0, \qquad (t,\rho)\in K\times[0,1].
\end{displaymath}
It follows that $A_\pm(\,\cdot\,,\rho)^{-\theta/\alpha}$ are holomorphic on $\Omega$ and satisfy $|A_\pm(t,\rho)^{-\theta/\alpha}|\leq c_K^{-\theta/\alpha}$ uniformly on $K\times[0,1]$.
Evaluating $\Psi$ by its defining integral, the right-hand side of \eqref{eq:variance-beta-representation} is therefore
\begin{displaymath}
R(t):=\frac{1}{2B(1/2,p)}\int_0^1\rho^{-1/2}(1-\rho)^{p-1} \bigl[A_+(t,\rho)^{-\theta/\alpha} +A_-(t,\rho)^{-\theta/\alpha}\bigr]\ud\rho.
\end{displaymath}
On each compact $K\subset\Omega$, its integrand is bounded in absolute value by a constant multiple of $\rho^{-1/2}(1-\rho)^{p-1}$, which is integrable because $p>0$.
Dominated convergence gives continuity of $R$, and another application of Fubini's theorem and Morera's theorem proves that $R$ is holomorphic on $\Omega$. The identity theorem now establishes
\eqref{eq:variance-beta-representation} throughout $\Omega$ and, in particular, for every $t>0$.

We now deduce \eqref{eq:variance-CR-introduction} for $z>0$, under the current bounded-support assumption.
Let $z>0$, and set $t=1/z$. Since $\mathcal T_p(z)=z^{-p}\EE_{(\alpha,\theta)}\left[\left(1+\frac{\V}{z}\right)^{-p}\right]$, equation~\eqref{eq:variance-beta-representation} gives
\begin{align}
\mathcal T_p(z)&=\frac{z^{-p}}{2 B(1/2,p)}\int_0^1\rho^{-1/2}(1-\rho)^{p-1} \nonumber\\
&\quad\times\left[\left\{\int_0^1\left[1+2\sqrt{\frac{\rho(1-\rho)}{z}}\,x+\frac{1-\rho}{z}x^2\right]^\alpha\G(\ud x)\right\}^{-\theta/\alpha}\right.\nonumber\\
&\qquad\left.+\left\{\int_0^1\left[1-2\sqrt{\frac{\rho(1-\rho)}{z}}\,x+\frac{1-\rho}{z}x^2\right]^\alpha\G(\ud x)\right\}^{-\theta/\alpha}\right]\ud\rho.
\label{eq:variance-beta-z}
\end{align}
Apply the change of variables $r=\sqrt{\frac{\rho}{1-\rho}}$, i.e. $\rho=\frac{r^2}{1+r^2}$ for $r>0$. Whence, $\rho^{-1/2}(1-\rho)^{p-1}\,\ud\rho=
2(1+r^2)^{-\theta}\,\ud r$. Moreover,
\begin{displaymath}
1\pm 2\sqrt{\frac{\rho(1-\rho)}{z}}\,x+\frac{1-\rho}{z}x^2=\frac{(x\pm\sqrt z\,r)^2+z}{z(1+r^2)}.
\end{displaymath}
Consequently, the factor $(1+r^2)^\theta$ arising from the power $-\theta/\alpha$ cancels with $(1+r^2)^{-\theta}$, whereas $z^{-p}z^\theta=z^{1/2}$. It follows that
\begin{align*}
\mathcal T_p(z)&=\frac{\sqrt z}{B(1/2,p)}\int_0^{+\infty}\left\{\int_0^1\bigl[(x+\sqrt z\,r)^2+z\bigr]^\alpha\,\G(\ud x)\right\}^{-\theta/\alpha}\ud r\\
&\quad+\frac{\sqrt z}{B(1/2,p)}\int_0^{+\infty}\left\{\int_0^1\bigl[(x-\sqrt z\,r)^2+z\bigr]^\alpha\,\G(\ud x)\right\}^{-\theta/\alpha}\ud r.
\end{align*}
Applying the changes of variables $s=-\sqrt z\,r$ in the first integral and $s=\sqrt z\,r$ in the second, and then combining the two resulting integrals, yields exactly \eqref{eq:variance-CR-introduction}. 

\paragraph{\normalfont\itshape Extension to general base probability measures.}
We now establish \eqref{eq:variance-CR-introduction} for a general base probability measure on $[0,+\infty)$ by a direct argument that also removes the bounded-support restriction. Let $\G$ be a probability measure on $[0,+\infty)$ satisfying $\int_0^{+\infty}x^{2\alpha}\,\G(\ud x)<+\infty$. By the integrability criterion for Pitman--Yor linear functionals established above, applied to the function $x\mapsto x^2$, this condition ensures that $\M_2:=\int_0^{+\infty}x^2\,\ptilde_{\alpha,\theta}(\ud x)<+\infty$ almost surely. Moreover, by the Cauchy--Schwarz inequality, $\M_1^2=\left(\int_0^{+\infty}x\,\ptilde_{\alpha,\theta}(\ud x)\right)^2\leq\int_0^{+\infty}x^2\,\ptilde_{\alpha,\theta}(\ud x)=\M_2$, and hence $\M_1:=\int_0^{+\infty}x\,\ptilde_{\alpha,\theta}(\ud x)<+\infty$ almost surely as well. Therefore, $\V=\M_2-\M_1^2$ is well defined and finite almost surely. For every $s\in\R$,
\begin{equation}\label{eq:variance-shift}
\int_0^{+\infty}(x-s)^2\,\ptilde_{\alpha,\theta}(\ud x)=\V+(s-\M_1)^2.
\end{equation}
Moreover, for every $a>0$,
\begin{equation}\label{eq:beta-shift-variance}
\int_{\R}(a+s^2)^{-\theta}\,\ud s=B\left(\frac12,p\right)a^{-p}.
\end{equation}
For $z>0$, applying \eqref{eq:beta-shift-variance} with $a=z+\V$ and translating the integration variable by $\M_1$, we obtain almost surely
\begin{align*}
B\left(\frac12,p\right)(z+\V)^{-p}&=\int_{\R}\left[z+\V+(s-\M_1)^2\right]^{-\theta}\ud s=\int_{\R}\left[z+\int_0^{+\infty}(x-s)^2\,\ptilde_{\alpha,\theta}(\ud x)\right]^{-\theta}\ud s.
\end{align*}
Since the integrand is non-negative, Tonelli's theorem gives
\begin{align*}
B\left(\frac12,p\right)\mathcal T_p(z)&=\int_{\R}\EE_{(\alpha,\theta)}\left[\left(z+\int_0^{+\infty}(x-s)^2\,\ptilde_{\alpha,\theta}(\ud x)\right)^{-\theta}\right]\ud s.
\end{align*}
For every fixed $s\in\R$, the inequality $|x-s|^{2\alpha}\leq C_\alpha\left(x^{2\alpha}+|s|^{2\alpha}\right)$ for some finite constant $C_\alpha$ implies that $\int_0^{+\infty}|x-s|^{2\alpha}\,\G(\ud x)<+\infty$. Therefore, the Tsilevich identity \eqref{eq:CR-introduction}, applied to the non-negative function $x\mapsto(x-s)^2$, yields
\begin{displaymath}
\EE_{(\alpha,\theta)}\left[\left(z+\int_0^{+\infty}(x-s)^2\,\ptilde_{\alpha,\theta}(\ud x)\right)^{-\theta}\right]=\left\{\int_0^{+\infty}\bigl[(x-s)^2+z\bigr]^\alpha\,\G(\ud x)\right\}^{-\theta/\alpha}.
\end{displaymath}
Hence
\begin{displaymath}
\mathcal T_p(z)=\frac{1}{B(1/2,p)}\int_{\R}\left\{\int_0^{+\infty}\bigl[(x-s)^2+z\bigr]^\alpha\,\G(\ud x)\right\}^{-\theta/\alpha}
\ud s,
\end{displaymath}
which proves \eqref{eq:variance-CR-introduction} for general $\G$ and $z>0$.

\paragraph{\normalfont\itshape The Dirichlet case.}
We now consider the Dirichlet case $\alpha=0$. Assume that $\int_0^{+\infty}\log(1+x^2)\,\G(\ud x)<+\infty$.  The corresponding integrability criterion for Dirichlet linear functionals ensures that $\M_2:=\int_0^{+\infty}x^2\,\ptilde_{\theta}(\ud x)<+\infty$ almost surely. Moreover, setting $\M_1:=\int_0^{+\infty}x\,\ptilde_{\theta}(\ud x)$, the Cauchy--Schwarz inequality gives $\M_1^2\leq\int_0^{+\infty}x^2\,\ptilde_{\theta}(\ud x)=\M_2$, and hence $\M_1<+\infty$ almost surely as well. Therefore, $\V=\M_2-\M_1^2$ is well defined and finite almost surely. The analogue of \eqref{eq:variance-shift} for the Dirichlet process, together with \eqref{eq:beta-shift-variance}, gives
\begin{align*}
B\left(\frac12,p\right)\mathcal T_p(z)&=\int_{\R}\EE_{(0,\theta)}\left[\left(z+\int_0^{+\infty}(x-s)^2\,\ptilde_{\theta}(\ud x)\right)^{-\theta}\right]\ud s.
\end{align*}
For every fixed $s\in\R$, since $1+(x-s)^2\leq C_s(1+x^2)$ for some finite constant $C_s$, the assumption on $\G$ implies that $\int_0^{+\infty}\log\left\{1+(x-s)^2\right\}\,\G(\ud x)<+\infty$. Thus, by the Cifarelli--Regazzini identity \eqref{eq:cf_indentity},
\begin{displaymath}
\EE_{(0,\theta)}\left[\left(z+\int_0^{+\infty}(x-s)^2\,\ptilde_{\theta}(\ud x)\right)^{-\theta}\right]=\exp\left\{-\theta\int_0^{+\infty}\log\bigl((x-s)^2+z\bigr)\,\G(\ud x)\right\}.
\end{displaymath}
It follows that
\begin{displaymath}
\mathcal T_p(z)=\frac{1}{B(1/2,p)}\int_{\R}\exp\left\{-\theta\int_0^{+\infty}\log\bigl((x-s)^2+z\bigr)\,\G(\ud x)\right\}\ud s,
\end{displaymath}
which proves \eqref{eq:variance-CR-introduction-DP} for $z>0$.

\paragraph{\normalfont\itshape Extension to the complex domain.}
We now extend both identities from $z>0$ to $\C\setminus(-\infty,0]$. Throughout this argument, $\ptilde$ denotes either $\ptilde_{\alpha,\theta}$, with $\alpha\in(0,1)$, or the Dirichlet process $\ptilde_\theta$, under the corresponding integrability assumption in Theorem~\ref{thm:CR_variance}. The expectation $\EE$ is taken under the corresponding law. In both cases, $\M_2<+\infty$ almost surely and
\begin{displaymath}
A_s:=\int_0^{+\infty}(x-s)^2\,\ptilde(\ud x)=\V+(s-\M_1)^2,
\qquad s\in\R.
\end{displaymath}
For every compact set $K\subset\C\setminus(-\infty,0]$, there exists $c_K>0$ such that $|z+a|\geq c_K(1+a),\qquad z\in K,\quad a\geq0$. Indeed, the ratio $|z+a|/(1+a)$ converges to $1$ as $a\to+\infty$, uniformly for $z\in K$, and has a strictly positive minimum on $K\times[0,R]$ for every finite $R$. Consequently, with all powers taken in their principal determination, $|(z+A_s)^{-\theta}|\leq c_K^{-\theta}(1+A_s)^{-\theta}$ for $z\in K$. Moreover, Tonelli's theorem and \eqref{eq:beta-shift-variance} give
\begin{align*}
\int_{\R}\EE_{(\alpha, \theta)}[(1+A_s)^{-\theta}]\,\ud s
&=\EE_{(\alpha, \theta)}\left[\int_{\R}\{1+\V+(s-\M_1)^2\}^{-\theta}\,\ud s\right]\\
&=B(1/2,p)\,\EE_{(\alpha, \theta)}[(1+\V)^{-p}]\leq B(1/2,p)<+\infty.
\end{align*}
Thus
\begin{displaymath}
H(z):=\frac{1}{B(1/2,p)}\int_{\R}\EE[(z+A_s)^{-\theta}]\,\ud s
\end{displaymath}
is well defined and holomorphic on $\C\setminus(-\infty,0]$, by holomorphy under an integral with a locally uniform integrable bound. The same estimate with $a=\V$ shows that $\mathcal T_p(z)=\EE[(z+\V)^{-p}]$ is holomorphic on $\C\setminus(-\infty,0]$. The preceding argument establishes $H(z)=\mathcal T_p(z)$ for $z>0$. Since $\C\setminus(-\infty,0]$ is connected, the identity theorem yields this equality throughout $\C\setminus(-\infty,0]$.

It remains to identify the integrand defining $H$. For $\alpha\in(0,1)$, the Tsilevich identity applied to $x\mapsto(x-s)^2$ gives
\begin{displaymath}
\EE_{(\alpha,\theta)}[(z+A_s)^{-\theta}]=\left\{\int_0^{+\infty}[(x-s)^2+z]^\alpha\,\G(\ud x)\right\}^{-\theta/\alpha},\qquad z\in\C\setminus(-\infty,0].
\end{displaymath}
Its integrability assumption follows from $\int_0^{+\infty}x^{2\alpha}\,\G(\ud x)<+\infty$. In the Dirichlet case, the corresponding Cifarelli--Regazzini identity gives
\begin{displaymath}
\EE_{(0,\theta)}[(z+A_s)^{-\theta}]=\exp\left\{-\theta\int_0^{+\infty}\operatorname{Log}((x-s)^2+z)\,\G(\ud x)\right\},\qquad z\in\C\setminus(-\infty,0].
\end{displaymath}
Here $\operatorname{Log}$ denotes the principal logarithm. For fixed $s$, the logarithmic moment assumption ensures that the integral on the right is holomorphic on $\C\setminus(-\infty,0]$: on each compact subset, its integrand is bounded in modulus by an integrable multiple of $1+\log(1+x^2)$. Thus the complex identity also follows from its positive-real version by the identity theorem. Substituting these expressions into $H$ proves \eqref{eq:variance-CR-introduction} and \eqref{eq:variance-CR-introduction-DP} throughout $\C\setminus(-\infty,0]$. This completes the proof.

\subsection{An interplay between Pitman--Yor and Dirichlet variances}

An interesting distributional connection between Pitman--Yor and Dirichlet means is known. In particular, \citet[Theorem~2.1]{Jam(08)} showed that the Pitman--Yor mean can be represented in distribution as a Dirichlet mean whose base probability measure is the distribution of the Pitman--Yor mean with parameters $(\alpha,0)$. The Cifarelli--Regazzini-type identities in Theorem~\ref{thm:CR_variance} reveal a natural counterpart of this correspondence for the Pitman--Yor variance. For $\alpha\in(0,1)$, let
\begin{displaymath}
\M_{\alpha,0}:=\int_0^{+\infty}x\,\ptilde_{\alpha,0}(\ud x),\qquad\V_{\alpha,0}:=\frac{1}{2}\int_0^{+\infty}\int_0^{+\infty}(x-y)^2\,\ptilde_{\alpha,0}(\ud x)\,\ptilde_{\alpha,0}(\ud y)
\end{displaymath}
be the mean and the variance of the Pitman--Yor process with parameters $(\alpha,0)$, respectively,  and denote by $Q_{\alpha,0}$ the distribution of $(\M_{\alpha,0},\V_{\alpha,0})$ on $[0,+\infty)^2$. The next proposition provides a counterpart of \citet[Theorem~2.1]{Jam(08)}, expressing the Pitman--Yor variance in distribution through a Dirichlet process whose base probability measure is precisely the distribution $Q_{\alpha,0}$.

\begin{prp}\label{prp:PY_DP_variance}
Let $\alpha\in(0,1)$ and $\theta>1/2$. Let $\ptilde_{\alpha,\theta}$ be a Pitman--Yor process on $[0,+\infty)$ with base probability measure $\G$ such that $\int_0^{+\infty}x^{2\alpha}\,\G(\ud x)<+\infty$, and $\ptilde_{0,\theta}^{(2)}$ be a Dirichlet process on $[0,+\infty)^2$ with parameter $\theta$ and base probability measure $Q_{\alpha,0}$. For the variance $\V$ of $\ptilde_{\alpha,\theta}$ defined in \eqref{eq:variance-introduction}, there holds 
\begin{equation}\label{eq:PY_DP_variance}
\V\stackrel{\mathrm d}{={}}\int_{[0,+\infty)^2}v\,\ptilde_{0,\theta}^{(2)}(\ud m,\ud v)+\frac{1}{2}\int_{[0,+\infty)^2}\int_{[0,+\infty)^2}(m-m')^2\,\ptilde_{0,\theta}^{(2)}(\ud m,\ud v)\,\ptilde_{0,\theta}^{(2)}(\ud m',\ud v').
\end{equation}
\end{prp}

\begin{proof}
For $s\in\R$, consider the linear functional $\int_0^{+\infty}(x-s)^2\,\ptilde_{\alpha,0}(\ud x)$. By the definitions of $\M_{\alpha,0}$ and $\V_{\alpha,0}$,
\begin{equation}\label{eq:variance_shift_identity}
\int_0^{+\infty}(x-s)^2\,\ptilde_{\alpha,0}(\ud x)=\V_{\alpha,0}+(\M_{\alpha,0}-s)^2.
\end{equation}
Moreover, for every $s\in\R$, $\int_0^{+\infty}|x-s|^{2\alpha}\,\G(\ud x)<+\infty$, as follows from the assumption $\int_0^{+\infty}x^{2\alpha}\,\G(\ud x)<+\infty$. Hence, \citet[Theorem~2.1]{Jam(08)} can be applied to the function $x\mapsto(x-s)^2$. Let $Q_{\alpha,0}^{(s)}$ denote the distribution of $\V_{\alpha,0}+(\M_{\alpha,0}-s)^2$. By \eqref{eq:variance_shift_identity},
$Q_{\alpha,0}^{(s)}$ is the image of $Q_{\alpha,0}$ under the map $(m,v)\mapsto v+(m-s)^2$. Therefore, by the image-measure property of the Dirichlet process and \citet[Theorem~2.1]{Jam(08)}, for $z>0$,
\begin{align}\label{eq:James_variance}
&\left\{\int_0^{+\infty}\bigl[(x-s)^2+z\bigr]^\alpha\,\G(\ud x)\right\}^{-\theta/\alpha}\\
&\notag\quad=\EE_{(0, \theta)}\left[\left\{z+\int_{[0,+\infty)^2}\bigl[v+(m-s)^2\bigr]\,\ptilde_{0,\theta}^{(2)}(\ud m,\ud v)\right\}^{-\theta}\right].
\end{align}
All the quantities appearing below are almost surely finite. Indeed, $\V_{\alpha,0}+\M_{\alpha,0}^2=\int_0^{+\infty}x^2\,\ptilde_{\alpha,0}(\ud x)$. Hence, \citet[Theorem~2.1]{Jam(08)}, applied to $x\mapsto x^2$, together with $\int_0^{+\infty}x^{2\alpha}\,\G(\ud x)<+\infty$, implies $\int_{[0,+\infty)^2}(v+m^2)\,\ptilde_{0,\theta}^{(2)}(\ud m,\ud v)<+\infty$ almost surely. Consequently,
\begin{displaymath}
\int_{[0,+\infty)^2}v\,\ptilde_{0,\theta}^{(2)}(\ud m,\ud v)<+\infty,\qquad\int_{[0,+\infty)^2}m^2\,\ptilde_{0,\theta}^{(2)}(\ud m,\ud v)<+\infty\qquad\text{a.s.},
\end{displaymath}
and the Cauchy--Schwarz inequality gives
\begin{displaymath}
\left\{\int_{[0,+\infty)^2}m\,\ptilde_{0,\theta}^{(2)}(\ud m,\ud v)\right\}^2\leq\int_{[0,+\infty)^2}m^2\,\ptilde_{0,\theta}^{(2)}(\ud m,\ud v)<+\infty\qquad\text{a.s.}
\end{displaymath}
Set
\begin{displaymath}
\overline{\M}:=\int_{[0,+\infty)^2}m\,\ptilde_{0,\theta}^{(2)}(\ud m,\ud v)
\end{displaymath}
and
\begin{displaymath}
W:=\int_{[0,+\infty)^2}v\,\ptilde_{0,\theta}^{(2)}(\ud m,\ud v)+\frac{1}{2}\int_{[0,+\infty)^2}\int_{[0,+\infty)^2}(m-m')^2\,\ptilde_{0,\theta}^{(2)}(\ud m,\ud v)\,\ptilde_{0,\theta}^{(2)}(\ud m',\ud v').
\end{displaymath}
By the usual representation of the variance,
\begin{displaymath}
W=\int_{[0,+\infty)^2}v\,\ptilde_{0,\theta}^{(2)}(\ud m,\ud v)+\int_{[0,+\infty)^2}m^2\,\ptilde_{0,\theta}^{(2)}(\ud m,\ud v)-\overline{\M}^{\,2}.
\end{displaymath}
It follows that
\begin{align*}
&\int_{[0,+\infty)^2}\bigl[v+(m-s)^2\bigr]\,\ptilde_{0,\theta}^{(2)}(\ud m,\ud v)\\
&\qquad=\int_{[0,+\infty)^2}v\,\ptilde_{0,\theta}^{(2)}(\ud m,\ud v)+\int_{[0,+\infty)^2}m^2\,\ptilde_{0,\theta}^{(2)}(\ud m,\ud v)-2s\overline{\M}+s^2\\
&\qquad=W+(\overline{\M}-s)^2.
\end{align*}
Hence \eqref{eq:James_variance} becomes
\begin{displaymath}
\left\{\int_0^{+\infty}\bigl[(x-s)^2+z\bigr]^\alpha\,\G(\ud x)\right\}^{-\theta/\alpha}=\EE_{(0, \theta)}\left[\bigl\{z+W+(\overline{\M}-s)^2\bigr\}^{-\theta}\right].
\end{displaymath}
Recall from Theorem~\ref{thm:CR_variance} that, with $p=\theta-1/2$,
\begin{align*}
\mathcal T_p(z)&=\frac{1}{B(1/2,p)}\int_{\R} \left\{\int_0^{+\infty}\bigl[(x-s)^2+z\bigr]^\alpha\,\G(\ud x)\right\}^{-\theta/\alpha}\ud s\\
&=\frac{1}{B(1/2,p)}\EE_{(0,\theta)}\left[\int_{\R} \bigl\{z+W+(\overline{\M}-s)^2\bigr\}^{-\theta}\ud s\right],
\end{align*}
where Tonelli's theorem justifies the interchange of expectation and integration. By translation invariance and the beta integral,
\begin{displaymath}
\int_{\R}\{a+(\overline{\M}-s)^2\}^{-\theta}\,\ud s=B(1/2,\theta-1/2)\,a^{-(\theta-1/2)},\qquad a>0.
\end{displaymath}
Since $p=\theta-1/2$, we obtain $\mathcal T_p(z)=\EE_{(0, \theta)}\left[\left(\frac{1}{z+W}\right)^p\right]$. On the other hand, by definition, $\mathcal T_p(z)=\EE_{(\alpha,\theta)}\left[\left(\frac{1}{z+\V}\right)^p
\right]$. Since the generalized Cauchy--Stieltjes transform of a nonnegative random variable uniquely determines its distribution, the preceding identity, holding for every $z>0$, implies $\V\stackrel{\mathrm d}{=}W$, which proves \eqref{eq:PY_DP_variance}.
\end{proof}

The two terms in distributional identity \eqref{eq:PY_DP_variance} have a direct interpretation. The first term is the mean of the second coordinate under $\ptilde_{0,\theta}^{(2)}$; by the image-measure property of the Dirichlet process, its marginal distribution is that of a Dirichlet mean whose base probability measure is the distribution of $\V_{\alpha,0}$. The second term is the variance of the first coordinate under $\ptilde_{0,\theta}^{(2)}$; marginally, it is therefore a Dirichlet variance whose base probability measure is the distribution of $\M_{\alpha,0}$. Thus, \eqref{eq:PY_DP_variance} has the same structure as the classical within--between decomposition of the variance: the first term represents an average of the within-component variances, whereas the second term represents the between-component variance of the corresponding means. Importantly, the two terms in \eqref{eq:PY_DP_variance} are constructed from the same Dirichlet process, whose base probability measure is the distribution of $(\M_{\alpha,0},\V_{\alpha,0})$, and are therefore not independent in general.

\subsection{The distribution function and density of the variance}\label{sec:inversion_variance}

We apply \citet[Theorem~6]{Byr(74)} to invert the generalized Cauchy--Stieltjes transforms in Theorem~\ref{thm:CR_variance}. The main task is to identify the boundary values of the transforms, and then to justify the passage to the limit in the inversion integral. Density formulas then follow by differentiation.

Throughout this section, set $p:=\theta-\frac12>0$, and take all complex powers and logarithms on their principal branches. For $\alpha\in(0,1)$, $\zeta\geq0$, and $s\in\R$, define
\begin{displaymath}
\mathcal A_\zeta(s):= \int_{\{|x-s|>\sqrt\zeta\}} \bigl((x-s)^2-\zeta\bigr)^\alpha\,\G(\ud x) + e^{-i\pi\alpha}\int_{\{|x-s|<\sqrt\zeta\}} \bigl(\zeta-(x-s)^2\bigr)^\alpha\,\G(\ud x).
\end{displaymath}
For the Dirichlet case, write
\begin{displaymath}
\Lambda_\zeta(s):=\int_0^{+\infty}\log\left|(x-s)^2-\zeta\right|\,\G(\ud x)
\end{displaymath}
and
\begin{displaymath}
M_\zeta(s):=\G\bigl(\{x\geq0:|x-s|<\sqrt\zeta\}\bigr).
\end{displaymath}
The following theorem provides the inversion of the Cifarelli--Regazzini-type identities in Theorem~\ref{thm:CR_variance}.

\begin{thm}\label{thm:inversion_variance}
Let $\alpha\in[0,1)$ and $\theta>1/2$, and let $\ptilde_{\alpha,\theta}$ be a Pitman--Yor process on $[0,+\infty)$ with base probability measure $\G$, where $\alpha=0$ corresponds to the Dirichlet process. Further, let $\V$ be the variance of $\ptilde_{\alpha,\theta}$ defined in \eqref{eq:variance-introduction}. Then the following statements hold.
\begin{itemize}
\item[i)] Let $\alpha\in(0,1)$. Assume that $\int_0^{+\infty}x^{2\alpha}\,\G(\ud x)<+\infty$ and, in addition, that $\G$ is not supported on a set of cardinality at most two. Then $\V$ has a continuous distribution function, given, for every $y>0$, by
\begin{equation}\label{eq:cdf_Var_general}
F_{\V}(y)=\frac{1}{\pi B(1/2,p)}
\int_0^y(y-\zeta)^{p-1}\int_{\R}
\Im\left\{\mathcal A_\zeta(s)^{-\theta/\alpha}\right\}\,\ud s\,\ud\zeta.
\end{equation}
If, in addition, $\theta>3/2$, its law is absolutely continuous, with a density which is continuous on $(0,+\infty)$, and given by
\begin{equation}\label{eq:density_Var_general}
f_{\V}(y)=\frac{p-1}{\pi B(1/2,p)}
\int_0^y(y-\zeta)^{p-2}\int_{\R}
\Im\left\{\mathcal A_\zeta(s)^{-\theta/\alpha}\right\}\,\ud s\,\ud\zeta.
\end{equation}
\item[ii)] Let $\alpha=0$. Assume that $\int_0^{+\infty}\log(1+x^2)\,\G(\ud x)<+\infty$ and, in addition, that $C_{\G}:=\sup_{a\in\R}\int_{\{|x-a|<1\}} \log\frac1{|x-a|}\,\G(\ud x)<+\infty$. Then $\V$ has a continuous distribution function, given, for every $y>0$, by
\begin{equation}\label{eq:cdf_Var_DP_general}
F_{\V}^{(0)}(y)=\frac{1}{\pi B(1/2,p)}
\int_0^y(y-\zeta)^{p-1}\int_{\R}
e^{-\theta\Lambda_\zeta(s)}\sin\{\pi\theta M_\zeta(s)\}\,\ud s\,\ud\zeta.
\end{equation}
If, in addition, $\theta>3/2$, its law is absolutely continuous, with a density which is continuous on $(0,+\infty)$, and given by
\begin{equation}\label{eq:density_Var_DP_general}
f_{\V}^{(0)}(y)=\frac{p-1}{\pi B(1/2,p)}
\int_0^y(y-\zeta)^{p-2}\int_{\R}
e^{-\theta\Lambda_\zeta(s)}\sin\{\pi\theta M_\zeta(s)\}\,\ud s\,\ud\zeta.
\end{equation}
\end{itemize}
When $\theta=3/2$, both laws are also absolutely continuous. Their densities are given, for almost every $y>0$, by the respective inner integrals in \eqref{eq:cdf_Var_general} and \eqref{eq:cdf_Var_DP_general}, evaluated at $\zeta=y$ and divided by $\pi B(1/2,1)$.
\end{thm}

\begin{proof}
In either case, $\V$ is almost surely finite under the assumptions of Theorem~\ref{thm:CR_variance}. Moreover, $\PP_{(\alpha,\theta)}(\V=0)=0$. Indeed, by the series representation \eqref{stick}, zero variance would require all the independent atom locations to coincide, an event of probability zero since $\G$ is not concentrated at one point. In case ii), the latter fact follows from $C_{\G}<+\infty$,
which implies that $\G$ is nonatomic. Write $F_{\V}$ for the distribution function of $\V$ in either case, and
\begin{displaymath}
\mathcal T_p(z)=\int_{[0,+\infty)}(z+v)^{-p}\,\ud F_{\V}(v), \qquad z\in\C\setminus(-\infty,0].
\end{displaymath}
Its holomorphy follows immediately from Morera's theorem and Fubini's theorem: on every compact subset $K$ of the cut plane, $|(z+v)^{-p}|$ is bounded uniformly in $z\in K$ and $v\geq0$. Dominated convergence therefore gives continuity, and integration around any closed, $C^1$ contour in the cut plane can be interchanged with integration against $\ud F$, yielding zero by Cauchy's theorem. Since $F_{\V}(0)=0$, \citet[Theorem~6]{Byr(74)}, with their parameter $\lambda=1$, gives
\begin{equation}\label{eq:BL_variance_inversion}
\frac{F_{\V}(y)+F_{\V}(y-)}{2}
=\frac1\pi\lim_{\epsilon\downarrow0} \int_{-y}^{1}(y+r)^{p-1}\Im\mathcal T_p(r-i\epsilon)\,\ud r, \qquad y>0.
\end{equation}
Here, we used $\mathcal T_p(\overline z)=\overline{\mathcal T_p(z)}$. The hypotheses of the inversion theorem hold because $F_{\V}$ has bounded variation and the defining Stieltjes integral converges absolutely. We now establish the boundary values and the uniform bounds needed in \eqref{eq:BL_variance_inversion}.

\smallskip
\noindent\emph{The Pitman--Yor case.}
For $\epsilon>0$, $\zeta\in\R$, and $s\in\R$, set
\begin{displaymath}
J_\epsilon(\zeta,s):=\int_0^{+\infty}
\bigl[(x-s)^2-\zeta-i\epsilon\bigr]^\alpha\,\G(\ud x),
\qquad
D_\zeta(s):=\int_0^{+\infty}|(x-s)^2-\zeta|^\alpha\,\G(\ud x).
\end{displaymath}
The arguments of the integrands defining $J_\epsilon$ lie in $(-\pi\alpha,0)$. Rotation by $\pi\alpha/2$ thus yields
\begin{displaymath}
|J_\epsilon(\zeta,s)|\geq c_\alpha \int_0^{+\infty}|(x-s)^2-\zeta-i\epsilon|^\alpha\,\G(\ud x) \geq c_\alpha D_\zeta(s), \qquad c_\alpha:=\cos(\pi\alpha/2)>0.
\end{displaymath}
The moment assumption makes $D_\zeta(s)$ jointly continuous. It is strictly positive: a zero would force $\G$ to be supported on the at most two solutions of $(x-s)^2=\zeta$. Hence $D$ has a positive lower bound on every compact set. To control large $|s|$, choose $R<+\infty$ such that $m:=\G([0,R])>0$. For every $Y>0$, there exists $S_Y>0$ such that, whenever $|s|\geq S_Y$, $|\zeta|\leq Y$, and $x\in[0,R]$, one has
\begin{displaymath}
|(x-s)^2-\zeta|\geq\frac{s^2}{4}.
\end{displaymath}
Consequently,
\begin{displaymath}
D_\zeta(s)\geq m\,4^{-\alpha}|s|^{2\alpha}, \qquad |s|\geq S_Y,\quad |\zeta|\leq Y.
\end{displaymath}
On the complementary compact set, joint continuity and strict positivity of $D_\zeta(s)$ provide a positive lower bound. Combining these estimates
with the preceding sector bound yields, for every $Y>0$,
\begin{equation}\label{eq:PY_boundary_domination}
|J_\epsilon(\zeta,s)^{-\theta/\alpha}|
\leq C_Y(1+s^2)^{-\theta},
\qquad |\zeta|\leq Y,\quad 0<\epsilon\leq1.
\end{equation}
This bound is integrable in $s$ since $\theta>1/2$. For $\zeta\geq0$, dominated convergence gives $J_\epsilon(\zeta,s)\to\mathcal A_\zeta(s)\neq0$; for $\zeta<0$, the limit is positive real. Applying \eqref{eq:PY_boundary_domination} to the representation in Theorem~\ref{thm:CR_variance}, we obtain
\begin{displaymath}
b_p(\zeta):=\frac1\pi\Im\mathcal T_p(-\zeta-i0) =\frac{1}{\pi B(1/2,p)}\int_{\R} \Im\left\{\mathcal A_\zeta(s)^{-\theta/\alpha}\right\}\,\ud s, \qquad \zeta\geq0.
\end{displaymath}
The same bound shows that $\mathcal T_p(-\zeta-i\epsilon)$ is uniformly bounded for $\zeta$ in compact intervals and $0<\epsilon\leq1$. In particular, $b_p$ is locally bounded, and the boundary imaginary part vanishes for $\zeta<0$.

\smallskip
\noindent\emph{The Dirichlet case.}
Define
\begin{displaymath}
L_{\G}(a):=\int_0^{+\infty}\log|x-a|\,\G(\ud x), \qquad L_{\zeta,\epsilon}(s):=\int_0^{+\infty} \operatorname{Log}\bigl((x-s)^2-\zeta-i\epsilon\bigr)\,\G(\ud x).
\end{displaymath}
The assumptions ensure that $L_{\G}(a)$ is finite for every $a\in\R$ and bounded below by $-C_{\G}$. For $\zeta\geq0$,
\begin{displaymath}
\Lambda_\zeta(s)=L_{\G}(s-\sqrt\zeta)+L_{\G}(s+\sqrt\zeta).
\end{displaymath}
Writing $q=(x-s)^2-\zeta$, we have, for $0<\epsilon\leq1$,
\begin{displaymath}
\bigl|\log|q-i\epsilon|\bigr|
\leq\log^-|q|+\log(1+|q|).
\end{displaymath}
For fixed $(\zeta,s)$, this is a $\G$-integrable majorant: the negative part is controlled by the logarithmic singularities at $s\pm\sqrt\zeta$ when $\zeta\geq0$, and is bounded when $\zeta<0$; the positive part is controlled by the logarithmic moment assumption. The imaginary parts are bounded by $\pi$. Since $\G$ is nonatomic, dominated convergence gives
\begin{displaymath}
L_{\zeta,\epsilon}(s)\rightarrow
\Lambda_\zeta(s)-i\pi M_\zeta(s),\qquad \zeta\geq0,
\end{displaymath}
and a real limit when $\zeta<0$. Also, since $\log|q-i\epsilon|\geq\log|q|$, one has
\begin{displaymath}
\left|e^{-\theta L_{\zeta,\epsilon}(s)}\right| \leq
\begin{cases}
e^{-\theta\{L_{\G}(s-\sqrt\zeta)+L_{\G}(s+\sqrt\zeta)\}},&\zeta\geq0,\\
e^{-2\theta L_{\G}(s)},&\zeta<0.
\end{cases}
\end{displaymath}
Choose $R<+\infty$ so that $m:=\G([0,R])>1/(2\theta)$. For sufficiently large $|a|$, restricting the positive part of $L_{\G}(a)$ to $[0,R]$ gives
\begin{displaymath}
L_{\G}(a)\geq m\log(|a|/2)-C_{\G}.
\end{displaymath}
Together with $L_{\G}\geq-C_{\G}$, this proves that, for every $Y>0$,
\begin{equation}\label{eq:DP_boundary_domination}
\left|e^{-\theta L_{\zeta,\epsilon}(s)}\right|
\leq C_Y(1+|s|)^{-2\theta m},
\qquad |\zeta|\leq Y,\quad 0<\epsilon\leq1.
\end{equation}
The right-hand side is integrable in $s$. Dominated convergence in the Dirichlet representation of Theorem~\ref{thm:CR_variance} now gives
\begin{displaymath}
b_p(\zeta):=\frac1\pi\Im\mathcal T_p(-\zeta-i0)
=\frac{1}{\pi B(1/2,p)}\int_{\R}
e^{-\theta\Lambda_\zeta(s)}\sin\{\pi\theta M_\zeta(s)\}\,\ud s,
\qquad \zeta\geq0.
\end{displaymath}
As in case i), $\mathcal T_p(-\zeta-i\epsilon)$ is uniformly bounded for $\zeta$ in compact intervals and $0<\epsilon\leq1$, $b_p$ is locally bounded, and the boundary imaginary part vanishes for $\zeta<0$.

\smallskip
\noindent\emph{Inversion and regularity.}
In both cases the bounds just proved, together with $p>0$, justify dominated convergence in \eqref{eq:BL_variance_inversion}. Setting $\zeta=-r$, we obtain
\begin{displaymath}
\frac{F_{\V}(y)+F_{\V}(y-)}2=\int_0^y(y-\zeta)^{p-1}b_p(\zeta)\,\ud\zeta.
\end{displaymath}
Because $b_p$ is locally bounded and the kernel $t^{p-1}$ is locally integrable, the right-hand side is continuous in $y\geq0$. Taking limits through continuity points of $F_{\V}$ shows that $F_{\V}$ is itself continuous and equals this integral. Substituting the respective expressions for $b_p$ proves \eqref{eq:cdf_Var_general} and \eqref{eq:cdf_Var_DP_general}.
When $p>1$, Fubini's theorem gives
\begin{displaymath}
F_{\V}(y)=\int_0^y\left\{(p-1)\int_0^t(t-\zeta)^{p-2}b_p(\zeta)\,\ud\zeta\right\}\ud t.
\end{displaymath}
The expression inside the braces is continuous for $t>0$, again by local boundedness of $b_p$ and local integrability of the kernel. It is therefore a continuous density, yielding \eqref{eq:density_Var_general} and \eqref{eq:density_Var_DP_general}. If $p=1$, the identity $F(y)=\int_0^y b_1(\zeta)\,\ud\zeta$ gives the asserted density almost everywhere.
\end{proof}

As an example, consider the uniform base probability measure $\G$ on $[0,1]$. The assumptions of Theorem~\ref{thm:inversion_variance} are then satisfied both for $\alpha\in(0,1)$ and for $\alpha=0$. Moreover, since $\ptilde_{\alpha,\theta}$ is supported on $[0,1]$, the classical bound for the variance of a probability measure supported on $[0,1]$ gives $0\leq \V\leq\frac14$ almost surely. For every $\zeta>0$, if $s\notin[-\sqrt{\zeta},1+\sqrt{\zeta}]$, then $|x-s|>\sqrt{\zeta}$ for every $x\in[0,1]$. Hence, for the Pitman--Yor process, $\mathcal A_\zeta(s)$ is positive real, so that $\Im\left\{\mathcal A_\zeta(s)^{-\theta/\alpha}\right\}=0$, whereas, for the Dirichlet process, $M_\zeta(s)=0$, and therefore $e^{-\theta \Lambda_\zeta(s)}\sin\{\pi\theta M_\zeta(s)\}=0$. Consequently, for the uniform base probability measure, the inner integrals over $\R$ appearing in Theorem~\ref{thm:inversion_variance} can be restricted to $[-\sqrt{\zeta},1+\sqrt{\zeta}]$. In particular, for $\alpha\in(0,1)$ and $\theta > 3/2$, the density of $\V$ takes the form
\begin{equation}\label{eq:uniform_density_PY}
f_{\V}(y)=\frac{\theta-\frac32}{\pi B(1/2,\theta-1/2)}\int_0^y(y-\zeta)^{\theta-5/2}\int_{-\sqrt{\zeta}}^{1+\sqrt{\zeta}}\Im\left\{\mathcal A_\zeta(s)^{-\theta/\alpha}\right\}\,\ud s\,\ud\zeta,\qquad 0<y<\frac14.
\end{equation}
For $\alpha=0$, the corresponding density is
\begin{equation}\label{eq:uniform_density_DP}
f_{\V}(y)=\frac{\theta-\frac32}{\pi B(1/2,\theta-1/2)}\int_0^y(y-\zeta)^{\theta-5/2}\int_{-\sqrt{\zeta}}^{1+\sqrt{\zeta}}e^{-\theta \Lambda_\zeta(s)}\sin\{\pi\theta M_\zeta(s)\}\,\ud s\,\ud\zeta,\qquad 0<y<\frac14.
\end{equation}
These simplified density representations of the Pitman--Yor and Dirichlet variances will be used for the numerical evaluations in Section~\ref{sec_num}, while the corresponding distribution functions follow analogously from Theorem~\ref{thm:inversion_variance}.


\section{Numerical illustrations for the Pitman--Yor and Dirichlet variances}\label{sec_num}

In this section, we provide numerical illustrations of the distributions of the Pitman--Yor and Dirichlet variances and assess the analytic inversion formulas derived in Section~\ref{sec4}. We consider three numerical procedures for evaluating the density of $\V$. First, the numerical-inversion procedure numerically inverts the generalized Cauchy--Stieltjes transform of $\V$ given in Theorem~\ref{thm:CR_variance}. Second, the numerical-evaluation procedure directly evaluates the analytic inversion formulas of Theorem~\ref{thm:inversion_variance} by numerical integration. Finally, the simulation-based procedure approximates the density of $\V$ by Monte Carlo simulation based on the stick-breaking representation \eqref{stick} of the Pitman--Yor process. The comparison of these procedures serves two purposes. On the one hand, it provides a numerical validation of the analytic inversion developed in Section~\ref{sec4}. On the other hand, it illustrates the distributional behavior of $\V$ for different choices of the parameters $(\alpha,\theta)$. We consider separately the Pitman--Yor variance, with $\alpha\in(0,1)$, and the Dirichlet variance, with $\alpha=0$.

\subsection{Numerical procedures}

The numerical-inversion procedure is based on the Cifarelli--Regazzini-type identities of Theorem~\ref{thm:CR_variance}. Specifically, we numerically invert the generalized Cauchy--Stieltjes transforms of $\V$ given in \eqref{eq:variance-CR-introduction} and \eqref{eq:variance-CR-introduction-DP}. We first evaluate these transforms at $100$ logarithmically spaced points between $10^{-2}$ and $10^3$, using the corresponding representations in Theorem~\ref{thm:CR_variance}. We then partition the support $[0,1/4]$ into $180$ equally spaced intervals and approximate the unknown density by a piecewise constant function on this partition, thereby reducing the inversion to a finite-dimensional linear inverse problem. Since this problem is ill-conditioned, we compute a nonnegative regularized solution using a second-order smoothness penalty, with regularization parameter $\lambda=10^{-6}$, following standard regularization techniques for discrete ill-posed problems; see, e.g., \citet[Section 5.1]{Han(98)}. Normalization is incorporated into the regularized fit, and the resulting reconstructed density is subsequently normalized numerically to integrate to one.

The numerical-evaluation procedure is based on the analytic inversion formulas of Theorem~\ref{thm:inversion_variance}. We evaluate the resulting representations of the density given in \eqref{eq:uniform_density_PY} and \eqref{eq:uniform_density_DP} by numerical quadrature. For $\alpha\in(0,1)$, the main numerical difficulty arises from the boundary values of the fractional powers appearing in the inversion formula. For each value of the outer integration variable, we locate the relevant zero-crossings and partition the integration domains accordingly, so as to account explicitly for changes in the analytic form of the integrands, evaluating the complex powers on the principal branch consistently with Theorem~\ref{thm:inversion_variance}. The outer integral is evaluated adaptively over the resulting intervals, with additional refinement where needed, whereas the inner integral is evaluated on a piecewise grid refined at its natural boundary points. The innermost integration with respect to the uniform base probability measure is computed by combining an incomplete-Beta representation with Gauss--Legendre quadrature. For $\alpha=0$, we directly evaluate the corresponding Dirichlet inversion formula, exploiting the explicit form of the logarithmic boundary terms and using a change of variables to handle the integrable endpoint behavior of the outer integral, followed by Gauss--Legendre quadrature. The numerical settings were selected through convergence checks, successively increasing the relevant quadrature resolutions and integration accuracy until the resulting density was numerically stable.

The simulation-based procedure is based on a Monte Carlo approximation of the density of $\V$ using the stick-breaking representation of the Pitman--Yor process. For $\alpha\in(0,1)$, direct truncation of the stick-breaking representation can be inaccurate because the residual mass decays only polynomially in the truncation level, particularly for large $\alpha$. We therefore approximate $\V$ through the first two Pitman--Yor moments using a backward recursion in which the residual moments are replaced by their expectations. After truncating the stick-breaking representation at $K$ terms, we initialize this contribution using the first two moments of the base probability measure $\G$ and proceed backwards, recursively incorporating each atom and its Beta stick-breaking factor. The variance is then obtained as the second moment minus the square of the first. We use $K=2000$, selected through convergence checks showing negligible changes for larger truncation levels relative to Monte Carlo variability. For $\alpha=0$, the residual stick decreases exponentially fast, so direct truncation is efficient; we stop when the remaining mass is below $10^{-12}$ and account for its contribution through the first two moments of $\G$. For both cases, we generate $120{,}000$ independent realizations of $\V$ and approximate the resulting density by kernel estimation.

To assess the accuracy of the procedures, we consider several diagnostics. Let $H(t)$ denote the generalized Cauchy--Stieltjes transform evaluated directly from Theorem~\ref{thm:CR_variance}, and let $\widehat H(t)$ denote either the transform of the density reconstructed by the numerical-inversion procedure or the empirical transform computed directly from the Monte Carlo realizations of $\V$. At each of the $100$ logarithmically spaced values of $t$ between $10^{-2}$ and $10^3$, we define the relative transform error as
\begin{displaymath}
\operatorname{RTE}(t)=\frac{|\widehat H(t)-H(t)|}{|H(t)|}.
\end{displaymath}
For the numerical-inversion procedure, the RTE provides a diagnostic of how accurately the reconstructed density reproduces the original transform. For the simulation-based procedure, it provides an independent comparison between the empirical and theoretical transforms, without involving kernel density estimation. In addition, for the simulation-based procedure we compare the Monte Carlo mean of $\V$ with its exact value, while for the numerical-evaluation procedure we verify that the evaluated density has total mass close to one. For the simulation-based procedure, we additionally compare the discrepancy $\widehat H(t)-H(t)$ with the Monte Carlo standard error of $\widehat H(t)$ to assess whether the observed transform discrepancies are compatible with Monte Carlo variability.

\subsection{The Pitman--Yor variance}

We first consider the Pitman--Yor case $\alpha\in(0,1)$, taking the base probability measure $\G$ to be the uniform distribution on $[0,1]$. We consider $\alpha\in\{0.25,0.50,0.75\}$ and $\theta\in\{2,3,4,5\}$. For each of the $12$ parameter configurations, we evaluate the density of the Pitman--Yor variance $\V$ using the procedures described above. This provides a systematic comparison over parameter regimes in which the numerical evaluation becomes progressively more demanding, particularly as $\alpha$ increases.

Figures~\ref{fig:variance-alpha025}--\ref{fig:variance-alpha075} report the resulting density evaluations and the corresponding transform-domain diagnostics, as measured by the RTE. Overall, the three numerical procedures yield closely comparable representations of the density of $\V$ across the parameter configurations considered for $(\alpha,\theta)$. The close agreement between the numerical-evaluation and simulation-based procedures provides an independent validation of the numerical evaluation of the analytic inversion formula, while the transform-domain comparison provides a more direct independent diagnostic, since the empirical transform is computed directly from the Monte Carlo realizations and does not use either the Cifarelli--Regazzini-type identity or its analytic inversion. The numerical-inversion procedure also reproduces the main features of the density, although discrepancies become more visible as $\alpha$ increases. This behavior is consistent with the ill-conditioned nature of the inverse problem rather than with a discrepancy at the level of the generalized Cauchy--Stieltjes transform: the RTE for the numerical-inversion procedure is of order $10^{-5}$ across the configurations considered. The RTE for the simulation-based procedure provides a further independent check, with discrepancies consistent with the corresponding Monte Carlo variability. These comparisons therefore provide strong numerical support for the analytic inversion formula, while also illustrating the greater numerical difficulty of recovering the density directly from its generalized Cauchy--Stieltjes transform.

\begin{figure}[t]
    \centering
     \includegraphics[width=\textwidth]{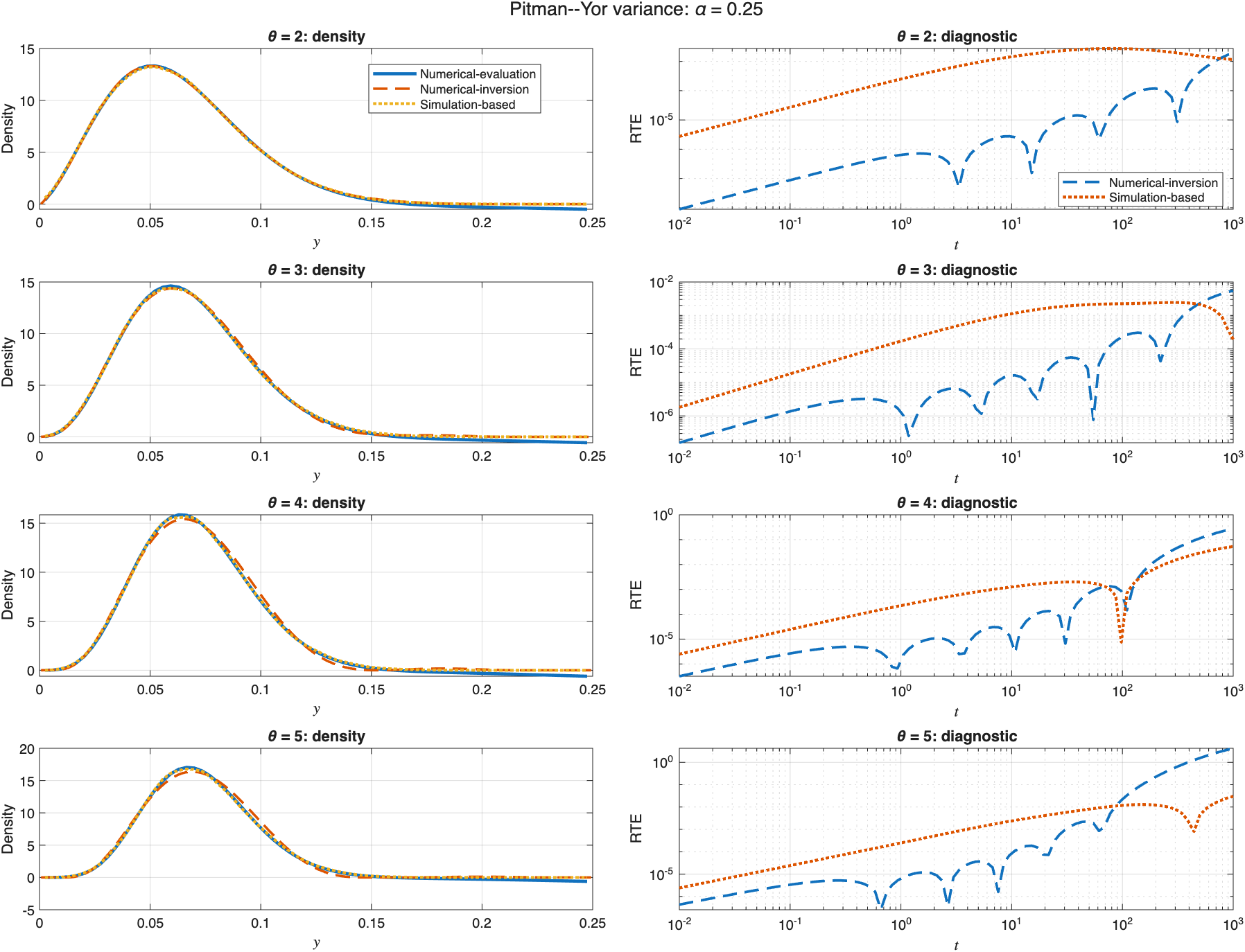}
\caption{\scriptsize{Density evaluations and transform-domain diagnostics for the Pitman--Yor variance $\V$ with $\alpha=0.25$, $\theta\in\{2,3,4,5\}$, and uniform base probability measure on $[0,1]$. Each row corresponds to a different value of $\theta$. The left panels compare the densities obtained from the numerical-inversion, numerical-evaluation, and simulation-based procedures, while the right panels report the RTE for the numerical-inversion and simulation-based procedures.}}
    \label{fig:variance-alpha025}
\end{figure}

\begin{figure}[t]
    \centering
     \includegraphics[width=\textwidth]{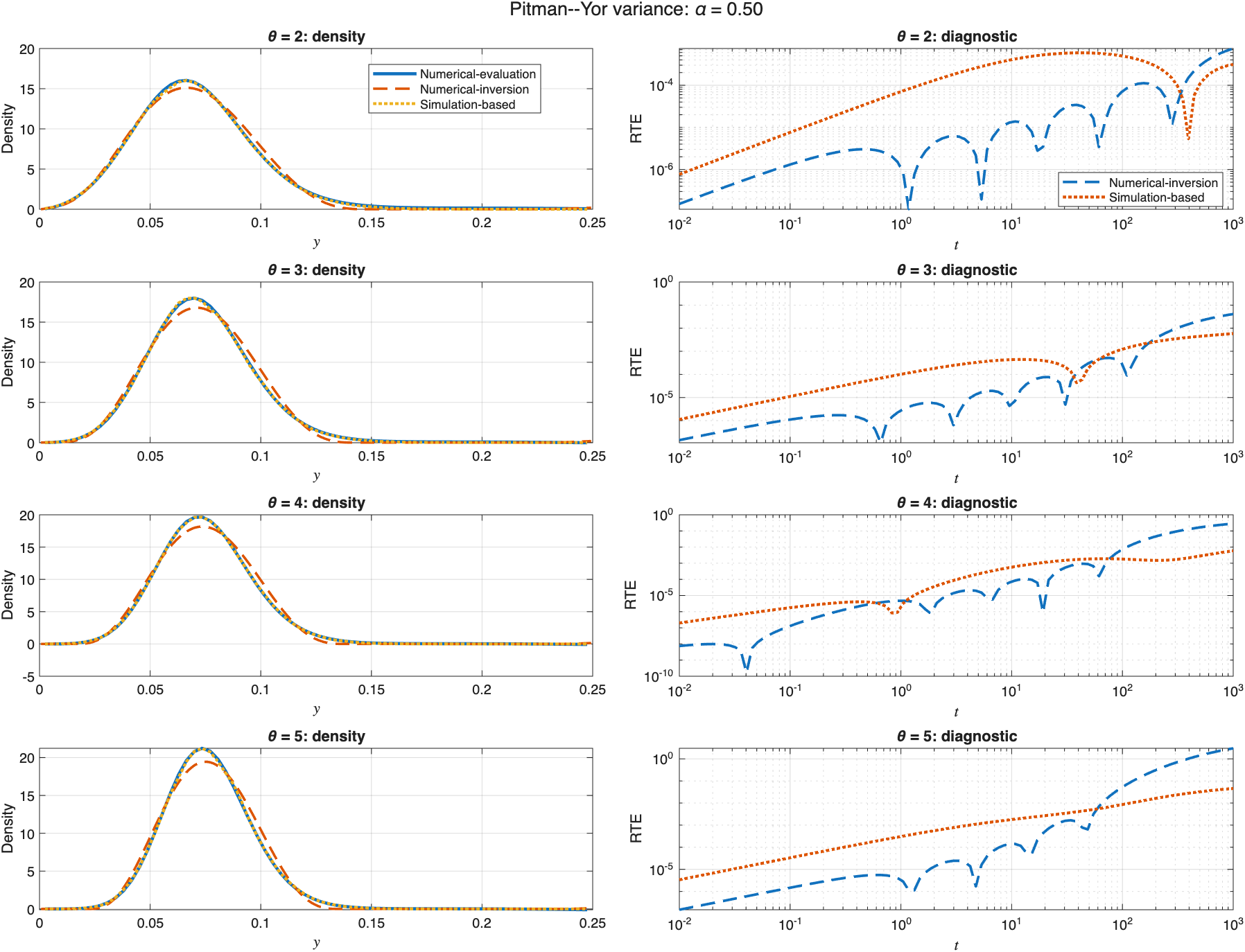}
\caption{\scriptsize{Density evaluations and transform-domain diagnostics for the Pitman--Yor variance $\V$ with $\alpha=0.50$, $\theta\in\{2,3,4,5\}$, and uniform base probability measure on $[0,1]$. Each row corresponds to a different value of $\theta$. The left panels compare the densities obtained from the numerical-inversion, numerical-evaluation, and simulation-based procedures, while the right panels report the RTE for the numerical-inversion and simulation-based procedures.}}    
\label{fig:variance-alpha050}
\end{figure}

\begin{figure}[t]
    \centering
     \includegraphics[width=\textwidth]{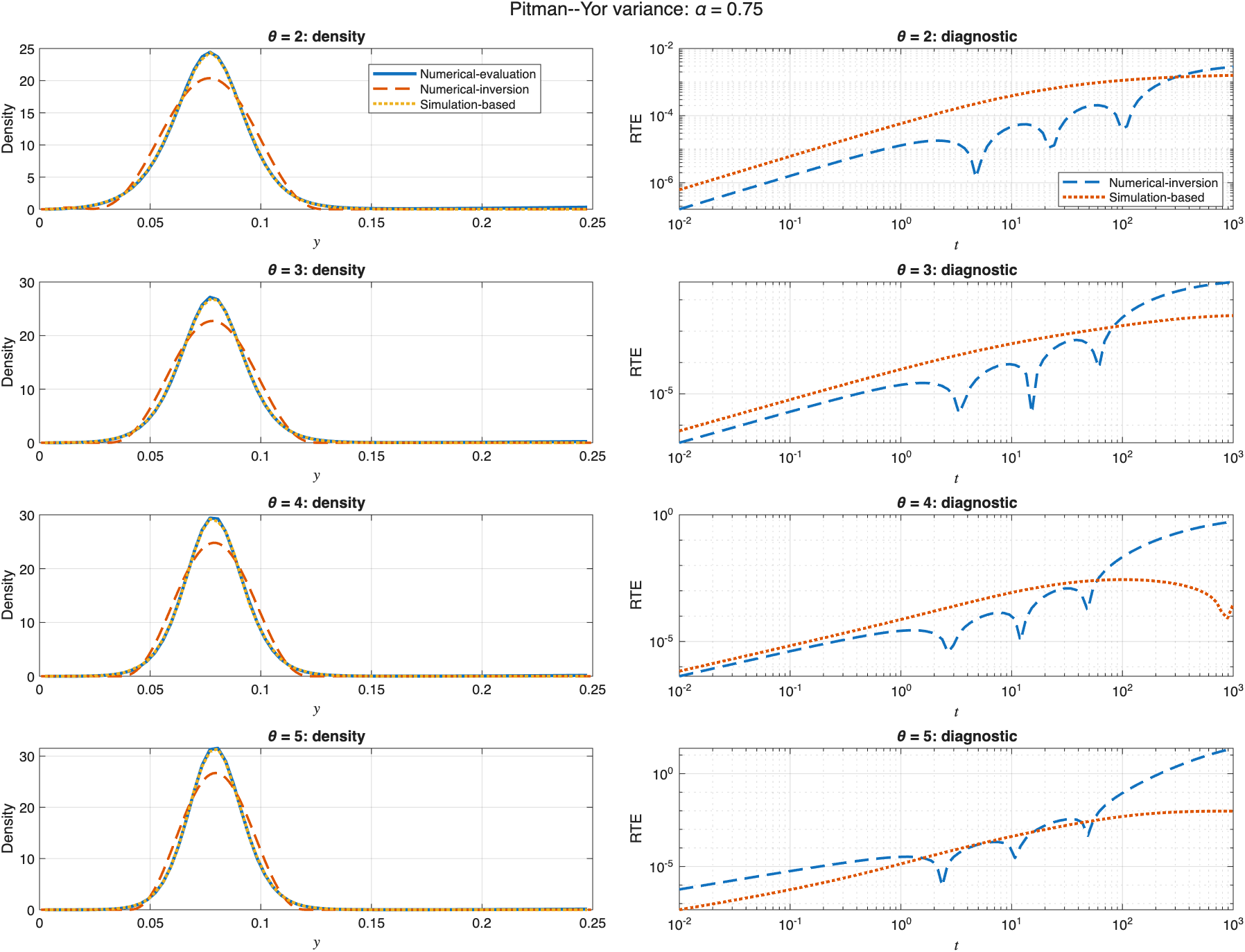}
\caption{\scriptsize{Density evaluations and transform-domain diagnostics for the Pitman--Yor variance $\V$ with $\alpha=0.75$, $\theta\in\{2,3,4,5\}$, and uniform base probability measure on $[0,1]$. Each row corresponds to a different value of $\theta$. The left panels compare the densities obtained from the numerical-inversion, numerical-evaluation, and simulation-based procedures, while the right panels report the RTE for the numerical-inversion and simulation-based procedures.}}    
\label{fig:variance-alpha075}
\end{figure}

\subsection{The Dirichlet variance}

We consider the Dirichlet case $\alpha=0$, again taking the base probability measure $\G$ to be the uniform distribution on $[0,1]$. We consider $\theta\in\{2,3,4,5\}$ as for the Pitman--Yor case, and, for each of the parameter configurations, we evaluate the density of the Dirichlet variance $\V$ using the three procedures described above.

Figure~\ref{fig:variance-alpha0} reports the resulting density evaluations and the corresponding transform-domain diagnostics, as measured by the RTE. The three numerical procedures show close agreement across the parameter configurations considered for $\theta$. In particular, the densities obtained by numerical evaluation of the analytic inversion formula have total mass very close to one, while the RTE for the numerical-inversion procedure is of order $10^{-5}$ or smaller. The simulation-based procedure provides a further independent validation: its RTE is consistent with Monte Carlo variability, and the simulated means of $\V$ are consistent with their exact values. Compared with the Pitman--Yor case with larger values of the parameter $\alpha$, the numerical inversion of the transform is also more accurate at the density level. Overall, these results provide strong numerical support for the analytic inversion formula for the Dirichlet variance established in Theorem~\ref{thm:inversion_variance}.

\begin{figure}[t]
    \centering
         \includegraphics[width=\textwidth]{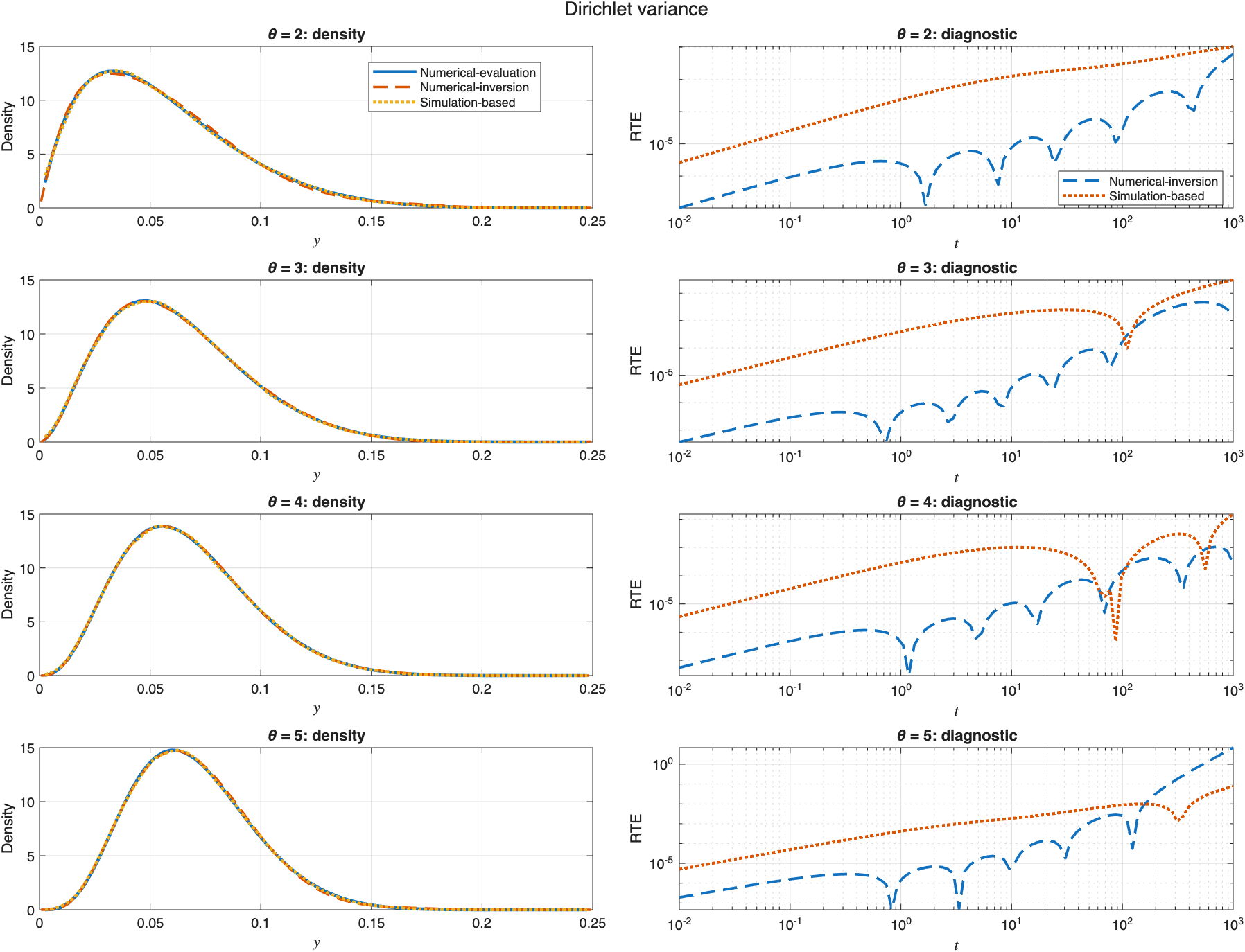}
\caption{\scriptsize{Density evaluations and transform-domain diagnostics for the Dirichlet variance $\V$ with $\theta\in\{2,3,4,5\}$ and uniform base probability measure on $[0,1]$. Each row corresponds to a different value of $\theta$. The left panels compare the densities obtained from the numerical-inversion, numerical-evaluation, and simulation-based procedures, while the right panels report the RTE for the numerical-inversion and simulation-based procedures.}}
    \label{fig:variance-alpha0}
\end{figure}

As an additional validation specific to the Dirichlet case, we consider the density representation of $\V$ derived by \citet[Theorem~3.1]{CifMel(00)}, which applies, in its stated form, when $\theta$ is a positive integer. Since this representation is obtained through an approach different from the one developed in Section~\ref{sec4}, it provides a useful benchmark for our analytic inversion formula. More precisely, letting $X$ and $Y$ be independent copies of the Dirichlet mean and denoting by $\varphi(\cdot,\theta)$ the density of $U=(X-Y)^2$, their result expresses the density of the Dirichlet variance $\V$ through an integral representation involving $\varphi(\cdot,\theta)$. Its numerical evaluation requires first determining the density $\varphi(\cdot,\theta)$ and then evaluating a derivative of order $\theta$ of an Abel-type integral involving $\varphi(\cdot,\theta)$. The latter step requires increasing numerical care as $\theta$ grows, due to the high-order differentiation involved. We implement this representation by first numerically evaluating the density of the corresponding Dirichlet mean and using it to obtain $\varphi(\cdot,\theta)$. The resulting function is evaluated by numerical quadrature and differentiated using Chebyshev spectral differentiation \citep[Chapter~6]{Tre00}, with the numerical resolution selected through stability checks. In contrast, the analytic inversion formula of Theorem~\ref{thm:inversion_variance} is directly amenable to numerical quadrature and requires neither the preliminary evaluation of Dirichlet mean densities and $\varphi(\cdot,\theta)$ nor high-order numerical differentiation, making its implementation more direct.

Figure~\ref{fig:melilli} compares the resulting density evaluations for $\theta\in\{2,3,4,5\}$. The two evaluations are in close agreement across all the parameter configurations considered for $\theta$. In particular, this agreement provides additional numerical support for the analytic inversion formula of Theorem~\ref{thm:inversion_variance}.

\begin{figure}[t]
    \centering
         \includegraphics[width=\textwidth]{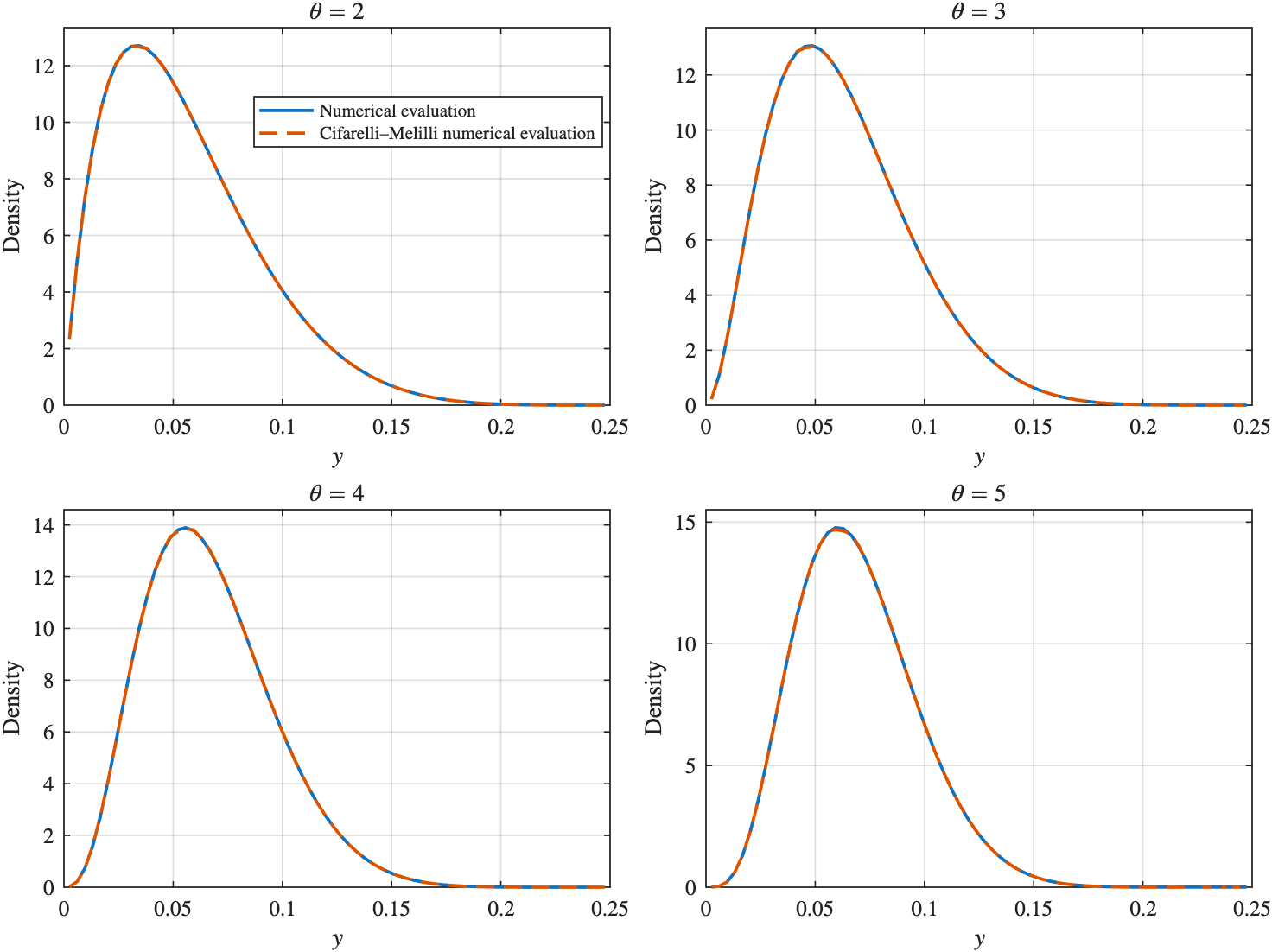}
\caption{\scriptsize{Comparison between the density of the Dirichlet variance $\V$ obtained by numerical evaluation of the analytic inversion formula in Theorem~\ref{thm:inversion_variance} and that obtained from the representation of \citet[Theorem~3.1]{CifMel(00)}, for $\theta\in\{2,3,4,5\}$ and uniform base probability measure $\G$ on $[0,1]$.}}
    \label{fig:melilli}
\end{figure}


\section{Concluding remarks}\label{sec5}

\subsection{Other quadratic functionals}

For $\alpha\in(0,1)$, $\theta>1/2$, and $\G$ supported on $[0,1]$, the moment-based argument of Section~\ref{sect:4-2} yields the beta-integral representation \eqref{eq:variance-beta-representation} of the generalized Cauchy--Stieltjes transform of the Pitman--Yor variance, i.e.
\begin{align*}
&\EE_{(\alpha,\theta)}\left[(1+t\V)^{-p}\right]\\
&\quad=\frac{1}{B(1/2,p)}\\
&\quad\quad\times\int_0^1\rho^{-1/2}(1-\rho)^{p-1}\frac{1}{2}\left[\Psi\left(2\sqrt{t\rho(1-\rho)},t(1-\rho)\right)\right.\left.+\Psi\left(-2\sqrt{t\rho(1-\rho)},t(1-\rho)\right)\right]\ud\rho.
\end{align*}
For $p=\theta-\tfrac12$, this is a representation of $\EE_{(\alpha,\theta)}[(1+t\V)^{-p}]$ in terms of a beta-weighted average of two symmetric evaluations of the function $\Psi$. Besides providing a route to the generalized Cauchy--Stieltjes transform of the Pitman--Yor variance, it shows how joint moments of linear functionals of the Pitman--Yor process can be used to obtain transform identities for quadratic functionals. These results suggest that the same approach may be useful for other quadratic functionals. 

The covariance provides a further application of this construction. In particular, for $\alpha\in[0,1)$ and $\theta>-\alpha$, let $\ptilde_{\alpha,\theta}$ be a Pitman--Yor process on $[0,1]^{2}$ with base probability measure $\G$. For $\mathbf{x}=(x_1,x_2)$, set
\begin{displaymath}
\M_{i,j}:=\int_{[0,1]^2}x_1^i x_2^j\, \ptilde_{\alpha,\theta}(\ud\mathbf{x}), \qquad (i,j)\in\{(1,0),(0,1),(1,1)\},
\end{displaymath}
and define the random covariance
\begin{equation}\label{eq_cova}
\CC:=\M_{1,1}-\M_{1,0}\M_{0,1}.
\end{equation}
To establish a generalized Cauchy--Stieltjes transform of $\CC$, consider the domain $\mathscr D:=\{(u,v,w)\in\C^3:\Re(1+u)>0,\,
\Re(1+v)>0,\,\Re(1+u+v+w)>0\}$. For every $(u,v,w)\in\mathscr D$ and $\mathbf{x}=(x_1,x_2)\in[0,1]^2$, one has $\Re(1+ux_1+vx_2+wx_1x_2)>0$.  Indeed, the real part of the expression is affine in each coordinate separately, and its minimum over $[0,1]^2$ is attained at a vertex. Consequently, the following joint generalized Cauchy--Stieltjes transforms are well defined on $\mathscr D$. For $\alpha\in(0,1)$ and $\theta>-\alpha$, set
\begin{equation}\label{eq:covariance-Psi-PY}
\Psi_{\mathrm{cov}}^{(\alpha)}(u,v,w) :=\left\{\int_{[0,1]^2} \bigl[1+ux_1+vx_2+wx_1x_2\bigr]^\alpha\, \G(\ud\mathbf{x})\right\}^{-\theta/\alpha},
\end{equation}
whereas for $\alpha=0$ and $\theta>0$ set
\begin{equation}\label{eq:covariance-Psi-DP}
\Psi_{\mathrm{cov}}^{(0)}(u,v,w) :=\exp\left\{-\theta\int_{[0,1]^2} \operatorname{Log}(1+ux_1+vx_2+wx_1x_2)\, \G(\ud\mathbf{x})\right\}.
\end{equation}
In both \eqref{eq:covariance-Psi-PY} and \eqref{eq:covariance-Psi-DP} complex powers and logarithms are understood in their principal determination.

\begin{prp}\label{prop:covariance-transform}
Let $\alpha\in[0,1)$ and $\theta>1$, and let $\ptilde_{\alpha,\theta}$ be a Pitman--Yor process on $[0,1]^2$ with base probability measure $\G$, where $\alpha=0$ corresponds to the Dirichlet process. Further, let $\CC$ be the covariance of $\ptilde_{\alpha,\theta}$ defined in \eqref{eq_cova}. Then, $|\CC|\leq1/4$ almost surely, and, for every $t \in (0,4)$,
\begin{equation}\label{eq:covariance-angular-transform}
\EE_{(\alpha,\theta)}\bigl[(1+t\CC)^{-(\theta-1)}\bigr] = \frac{\theta-1}{2\pi} \int_0^1(1-\rho)^{\theta-2}\int_0^{2\pi} \Psi_{\mathrm{cov}}^{(\alpha)}\!\left( a_{t,\rho}e^{i\varphi},a_{t,\rho}e^{-i\varphi},t(1-\rho)
\right)\,\ud\varphi\,\ud\rho
\end{equation}
holds with $a_{t,\rho}:=\sqrt{t\rho(1-\rho)}$. 
\end{prp}

The proof of Proposition~\ref{prop:covariance-transform} is given in Appendix~\ref{app:covariance-transform}. Expanding $\CC^n$ produces joint powers of three linear functionals. Euler's beta integral and an angular average then provide their mixed moments, with the angular integration selecting equal powers of its first two arguments. This replaces the symmetric two-point average appearing in \eqref{eq:variance-beta-representation} and leads to the transform order $\theta-1$. In particular, the boundedness of $\CC$ implies that \eqref{eq:covariance-angular-transform}, through its Taylor expansion at the origin, determines all moments and hence the distribution of the Pitman--Yor and Dirichlet covariances.

\subsection{Open problems}

A natural extension of the present work concerns higher-order moments and, more generally, polynomial functionals of Pitman--Yor processes. Unlike the variance, which depends only on the first two moments, these functionals may involve several moments and their products, leading to increasingly complex dependence structures. Their analysis therefore requires a systematic treatment of joint moments of linear functionals, for which higher-dimensional counterparts of the Tsilevich identity may provide a natural starting point. Such identities could potentially yield representations of generalized Cauchy--Stieltjes transforms for broader classes of polynomial functionals, extending the strategy developed here for the variance. However, the availability of a transform identity does not, by itself, guarantee that the corresponding distribution can be recovered explicitly. A central challenge is to identify suitable representations and analytic structures that make these transforms amenable to inversion, thereby obtaining distributional results rather than transform characterizations alone. Beyond Pitman--Yor processes, it would also be interesting to investigate whether analogous methods can be developed for other classes of random probability measures. In this direction, the availability of tractable joint transforms or combinatorial sampling structures may play a key role in extending the present approach to more general models.

A particularly interesting direction concerns the Gini mean difference, which provides a natural example of a nonlinear functional beyond the polynomial setting. In particular, for $\alpha\in[0,1)$ and $\theta>-\alpha$, let $\ptilde_{\alpha,\theta}$ be a Pitman--Yor process on $[0,1]$ with base probability measure $\G$, and consider
\begin{displaymath}
\DD:=\int_{[0,1]}\int_{[0,1]}|x-x'|\,\ptilde_{\alpha,\theta}(\ud x)\,\ptilde_{\alpha,\theta}(\ud x').
\end{displaymath}
Extending the transform-and-inversion strategy of \citet{CifReg(90)} to $\DD$ would require a suitable representation of its generalized Cauchy--Stieltjes transform. A first route relies on expressing $\DD$ as an integral of the variance of the Pitman--Yor process. More precisely, for $y\in[0,1]$, if
\begin{displaymath}
\psi_y(x):=\sqrt{2}\,\ind_{[0,x]}(y),\qquad\V_y:=\frac12\int_{[0,1]}\int_{[0,1]}\bigl[\psi_y(x)-\psi_y(x')\bigr]^2\,\ptilde_{\alpha,\theta}(\ud x)\,\ptilde_{\alpha,\theta}(\ud x')
\end{displaymath}
then
\begin{displaymath}
\DD=\int_0^1\V_y\,\ud y=\lim_{N\to+\infty}\frac1N\sum_{j=1}^N\V_{j/N}
\qquad\text{almost surely},
\end{displaymath}
with $0\leq\DD\leq1/2$. This representation of $\DD$ suggests studying the joint distributions of the random variances $\V_y$, requiring control of their dependence structure through joint transforms of increasing dimension. An alternative approach relies on an auxiliary Gaussian process. More precisely, for $z>1/2$, let $B$ be standard Brownian motion on $[0,1]$ and let $Z$ be a standard normal random variable, with $B$, $Z$, and $\ptilde_{\alpha,\theta}$ mutually independent. Then, the Gaussian process defined by
\begin{displaymath}
W_z(x):=\sqrt{2}\,B(x)-\frac{B(1)}{\sqrt{2}}+\sqrt{z-\tfrac12}\,Z
\end{displaymath}
is centered with covariance kernel $z-|x-x'|$, and satisfies
\begin{displaymath}
\left.\int_{[0,1]}W_z(x)\,\ptilde_{\alpha,\theta}(\ud x)\ \right|\ \ptilde_{\alpha,\theta}\sim\mathcal{N}(0,z-\DD),
\end{displaymath}
where $\mathcal{N}(\mu,\sigma^2)$ denotes the normal distribution with mean $\mu$ and variance $\sigma^2$. Conditioning instead on the Gaussian process suggests exploiting transform identities for linear Pitman--Yor functionals. The resulting representation, however, still involves an expectation over the Gaussian path, and obtaining an explicit expression for this expectation remains a substantial difficulty.

\section*{Acknowledgments}

The authors have been privileged to have Eugenio Regazzini as a teacher, a mentor, and a constant source of scientific inspiration for the past twenty years. His guidance has profoundly shaped their scientific paths and their way of thinking about and conducting research. They have also learned enormously from his work: his writings have accompanied their development as researchers and remain a source of ideas, perspective, and inspiration. Through his work and his example, Eugenio Regazzini has taught them not only a great deal about probability and statistics, but also the rigor, intellectual curiosity, and passion for the discipline that have guided them throughout their careers. They are deeply grateful for all they have learned from him, and it is a great privilege for them to dedicate this work to him on the occasion of his 80th birthday.


\appendix

\renewcommand{\theequation}{\Alph{section}.\arabic{equation}}

\setcounter{equation}{0}

\section{On linear functionals of the Pitman--Yor process}\label{app:linear_inversion}

The distributional theory of linear functionals of the Pitman--Yor process was studied by \citet{Jam(08)} through the analytic inversion of the Tsilevich identity. We focus here on the Pitman--Yor mean $\L:=\int_0^{+\infty} x\,\ptilde_{\alpha,\theta}(\ud x)$, for which \citet[Theorem 3.1]{Jam(08)} provides an inversion formula valid for $\alpha\in(0,1)$ and $\theta>0$. The resulting explicit representations of the density are given in different forms according to whether the parameter $\theta>1$, $\theta=1$, or $0<\theta<1$. We present these regimes in a unified form by differentiating the representation of the distribution function given in \citet[Equation~(24)]{Jam(08)}, which is valid for $\theta>0$. The result for a general linear functional $\L_\varphi=\int_{\R}\varphi(x)\,\ptilde_{\alpha,\theta}(\ud x)$, with $\varphi$ a non-negative measurable function, follows by applying the same argument to the image measure of the base probability measure $\G$ under $\varphi$.

Let $\alpha\in(0,1)$ and let $\G$ be a probability measure on $[0,+\infty)$ satisfying $\int_0^{+\infty}x^\alpha\,\G(\ud x)<+\infty$. Following \citet[Section 3]{Jam(08)}, define the generalized or fractional Abel transforms of $\G$ as
\begin{displaymath}
\mathcal{A}_{\G,\alpha}^+(t):=\int_t^{+\infty}(x-t)^\alpha\,\G(\ud x)\qquad\text{and}\qquad\mathcal{A}_{\G,\alpha}(t):=\int_0^t(t-x)^\alpha\,\G(\ud x),
\end{displaymath}
and introduce
\begin{displaymath}
\gamma_\alpha(t):=\cos(\alpha\pi)\mathcal{A}_{\G,\alpha}(t)+\mathcal{A}_{\G,\alpha}^+(t),
\end{displaymath}
\begin{displaymath}
\zeta_\alpha(t):=\sin(\alpha\pi)\mathcal{A}_{\G,\alpha}(t),
\end{displaymath}
\begin{displaymath}
R_\alpha(t):=\sqrt{\gamma_\alpha^2(t)+\zeta_\alpha^2(t)},
\end{displaymath}
and
\begin{displaymath}
\phi_\alpha(t):=\Arg\left\{\gamma_\alpha(t)+i\zeta_\alpha(t)\right\}\in[0,\alpha\pi]
\end{displaymath}
whenever $R_\alpha(t)>0$. Equivalently, $\phi_\alpha(t)=\arctan\left(\frac{\zeta_\alpha(t)}{\gamma_\alpha(t)}\right)+\pi\ind_{\{\gamma_\alpha(t)<0\}}$ whenever $\gamma_\alpha(t)\neq0$. Notice that, since $\mathcal A_{\G,\alpha}(t)\geq0$ and $\mathcal A_{\G,\alpha}^+(t)\geq0$, one has $R_\alpha(t)=0$ if and only if $\mathcal A_{\G,\alpha}(t)=\mathcal A_{\G,\alpha}^+(t)=0$. In particular, this is equivalent to $\G=\delta_t$ and, in that degenerate case, $\L=t$ almost surely. Accordingly, whenever $\G$ is non-degenerate, $R_\alpha(t)>0$ for every $t\geq0$.

\begin{prp}\label{prp:unified_inversion}
Let $\alpha\in(0,1)$ and $\theta>0$, and let $\ptilde_{\alpha,\theta}$ be a Pitman--Yor process on $[0,+\infty)$ with base probability measure $\G$ satisfying $\int_0^{+\infty}x^\alpha\,\G(\ud x)<+\infty$. Let $\L:=\int_0^{+\infty}x\,\ptilde_{\alpha,\theta}(\ud x)$, and assume that the distribution of $\L$ is absolutely continuous, with density $q_{\alpha,\theta}$. Then, for almost every $y>0$,
\begin{equation}\label{eq:JLP}
q_{\alpha,\theta}(y)=\frac{\ud}{\ud y}\int_0^y(y-t)^{\theta-1}\Delta_{\alpha,\theta}(t)\,\ud t,
\end{equation}
where
\begin{displaymath}
\Delta_{\alpha,\theta}(t)=\frac{1}{\pi[R_\alpha(t)]^{\theta/\alpha}}\sin\left(\frac{\theta}{\alpha}\phi_\alpha(t)\right).
\end{displaymath}
The derivative in \eqref{eq:JLP} is understood in the almost-everywhere sense, consistently with the assumed absolute continuity of the distribution
of $\L$. If $\Delta_{\alpha,\theta}$ is continuously differentiable and $\Delta_{\alpha,\theta}(0^+)$ exists and is finite, then \eqref{eq:JLP} can equivalently be written as
\begin{equation}\label{eq:JLP_extra}
q_{\alpha,\theta}(y)=\int_0^y(y-t)^{\theta-1}\Delta_{\alpha,\theta}'(t)\,\ud t+\Delta_{\alpha,\theta}(0^+)\,y^{\theta-1}.
\end{equation}
\end{prp}

\begin{proof}
Consider the complex domain $\mathbb C\setminus(-\infty,0]$ and, for $z\in\C\setminus(-\infty,0]$, define $I_\alpha(z):=\int_0^{+\infty}(z+x)^\alpha\,\G(\ud x)$ and $\mathcal S_\theta(z):=\EE_{(\alpha,\theta)}[(z+\L)^{-\theta}]$, with all complex powers taken in their principal determination. By the analytic properties of the Tsilevich identity, $I_\alpha$ is holomorphic on $\C\setminus(-\infty,0]$ and takes values in $\C\setminus(-\infty,0]$. Thus $\operatorname{Log}I_\alpha$ is well defined by the principal logarithm, and the Tsilevich identity gives
\begin{equation}\label{eq:analytic_CR_mean}
\mathcal S_\theta(z)=\exp\left\{-\frac{\theta}{\alpha}\operatorname{Log}I_\alpha(z)\right\},
\qquad z\in\C\setminus(-\infty,0].
\end{equation}
We now determine the lower boundary value on the negative real axis. Fix $y>0$ and let $z=-y-i\epsilon$, with $\epsilon>0$. Splitting the integral at $y$ gives
\begin{displaymath}
I_\alpha(-y-i\epsilon)=\int_{[0,y)}(x-y-i\epsilon)^\alpha\,\G(\ud x)+\int_{[y,+\infty)}(x-y-i\epsilon)^\alpha\,\G(\ud x).
\end{displaymath}
For $x>y$, $(x-y-i\epsilon)^\alpha\rightarrow(x-y)^\alpha$ as $\epsilon\downarrow0$, whereas for $x<y$, $(x-y-i\epsilon)^\alpha\rightarrow e^{-i\pi\alpha}(y-x)^\alpha$. The contribution at $x=y$ tends to zero as $\epsilon\downarrow0$. Moreover, for $0<\epsilon\leq1$, $|(x-y-i\epsilon)^\alpha|\leq C_y(1+x^\alpha)$ for some finite constant $C_y$, so dominated convergence applies.
Therefore
\begin{displaymath}
I_\alpha(-y-i0):=\lim_{\epsilon\downarrow0}I_\alpha(-y-i\epsilon)=\mathcal A_{\G,\alpha}^+(y)+e^{-i\pi\alpha}\mathcal A_{\G,\alpha}(y)=\gamma_\alpha(y)-i\zeta_\alpha(y).
\end{displaymath}
By the definitions of $R_\alpha$ and $\phi_\alpha$, $\gamma_\alpha(y)-i\zeta_\alpha(y)=R_\alpha(y)e^{-i\phi_\alpha(y)}$. Moreover, for $\Im z<0$, the argument of $I_\alpha(z)$ remains in $(-\alpha\pi,0)$. Hence the analytic determination $\operatorname{Log}I_\alpha$, normalized on the positive real axis, approaches the lower boundary without any additional $2\pi i$ winding. Therefore, $\operatorname{Log}I_\alpha(-y-i0)=\log R_\alpha(y)-i\phi_\alpha(y)$. Taking the lower boundary value in \eqref{eq:analytic_CR_mean} then gives
\begin{displaymath}
\mathcal S_\theta(-y-i0)=\exp\left\{-\frac{\theta}{\alpha}\left[\log R_\alpha(y)-i\phi_\alpha(y)\right]\right\}=[R_\alpha(y)]^{-\theta/\alpha}\exp\left\{i\frac{\theta}{\alpha}\phi_\alpha(y)
\right\}.
\end{displaymath}
Consequently,
\begin{equation}
\frac{1}{\pi}\Im\mathcal S_\theta(-y-i0)=\frac{1}{\pi[R_\alpha(y)]^{\theta/\alpha}}\sin\left(\frac{\theta}{\alpha}\phi_\alpha(y)\right)=\Delta_{\alpha,\theta}(y).
\label{eq:spectral_jump_mean}
\end{equation}
Let $Q_{\alpha,\theta}$ denote the probability distribution of $\L$. We apply \citet[Theorem~6]{Byr(74)}. Since the law of $\L$ is absolutely continuous, $\G$ cannot be concentrated at a single point. Hence $R_\alpha(t)>0$ for every $t\geq0$. For each fixed $y>0$, continuity gives $\inf_{0\leq t\leq y}R_\alpha(t)>0$. Moreover, dominated convergence gives joint continuity of $I_\alpha(-t-i\epsilon)$ in $(t,\epsilon)$ up to
$\epsilon=0$, with the lower boundary determination used at $\epsilon=0$. Consequently, $\sup_{0\leq t\leq y,\,0<\epsilon\leq\epsilon_0}
|\mathcal S_\theta(-t-i\epsilon)|<+\infty$ for some $\epsilon_0>0$. The analogous bound holds for the upper boundary values by complex conjugation.

In the inversion formula, the part of the integration contour corresponding to a bounded interval of the positive real axis has vanishing jump in the limit; the same local boundedness argument applies there, including at the origin. Since $(y-t)^{\theta-1}$ is integrable on $(0,y)$ for every $\theta>0$, dominated convergence therefore yields
\begin{displaymath}
Q_{\alpha,\theta}((-\infty,y])=\int_0^y(y-t)^{\theta-1}\frac{1}{\pi}\Im\mathcal S_\theta(-t-i0)\,\ud t.
\end{displaymath}
This is also the distribution-function representation \citet[Equation~(24)]{Jam(08)}. Using \eqref{eq:spectral_jump_mean}, this becomes $Q_{\alpha,\theta}((-\infty,y])=\int_0^y(y-t)^{\theta-1}\Delta_{\alpha,\theta}(t)\,\ud t$. Since $Q_{\alpha,\theta}$ is absolutely continuous by assumption, differentiating the preceding identity yields, for almost every $y>0$,
\begin{displaymath}
q_{\alpha,\theta}(y)=\frac{\ud}{\ud y}\int_0^y(y-t)^{\theta-1}\Delta_{\alpha,\theta}(t)\,\ud t,
\end{displaymath}
which is \eqref{eq:JLP}. We also note that $\Delta_{\alpha,\theta}$ is absolutely continuous on every compact interval $[0,Y]$. Indeed, the Abel transforms
$\mathcal A_{\G,\alpha}$ and $\mathcal A_{\G,\alpha}^{+}$ are absolutely continuous on $[0,Y]$: their almost-everywhere derivatives are obtained by differentiating the kernels, and absolute integrability follows from Tonelli's theorem and $\alpha-1>-1$. More explicitly, each derivative has $L^1(0,Y)$ norm at most $Y^\alpha$.
Since
\begin{displaymath}
\Delta_{\alpha,\theta}(t)=\frac{1}{\pi}\Im\left\{\bigl[\mathcal A_{\G,\alpha}^{+}(t)+e^{-i\pi\alpha}\mathcal A_{\G,\alpha}(t)\bigr]^{-\theta/\alpha}\right\}
\end{displaymath}
and the expression inside the power stays away from zero and from the branch cut on $[0,Y]$, the same absolute continuity property holds for $\Delta_{\alpha,\theta}$. If, in addition, $\Delta_{\alpha,\theta}$ is continuously differentiable and $\Delta_{\alpha,\theta}(0^+)$ exists and is finite, then
\begin{displaymath}
\Delta_{\alpha,\theta}(t)=\Delta_{\alpha,\theta}(0^+)+\int_0^t\Delta_{\alpha,\theta}'(r)\,\ud r.
\end{displaymath}
Substituting this expression into \eqref{eq:JLP}, applying Fubini's theorem, and differentiating with respect to $y$ yields \eqref{eq:JLP_extra}. This completes the proof.
\end{proof}

\section{Proof of Proposition~\ref{prop:covariance-transform}} \label{app:covariance-transform}

We use the notation of Section~\ref{sec5} and set $r:=\theta-1>0$.
For each realization of $\ptilde_{\alpha,\theta}$, the variance of either
coordinate is at most $1/4$. The Cauchy--Schwarz inequality therefore
gives $|\CC|\leq1/4$.

For $|u|+|v|+|w|<1$, the multivariate form of the Tsilevich identity,
or of the Cifarelli--Regazzini identity when $\alpha=0$, gives
\begin{equation}\label{eq:covariance-joint-transform}
\Psi_{\mathrm{cov}}^{(\alpha)}(u,v,w)
=\EE_{(\alpha,\theta)}\!\left[
\bigl(1+u\M_{1,0}+v\M_{0,1}+w\M_{1,1}\bigr)^{-\theta}
\right].
\end{equation}
Indeed, for nonnegative real $u,v,w$, this follows by applying the corresponding linear-functional identity to $\mathbf{x}\mapsto ux_1+vx_2+wx_1x_2$.  Both sides are holomorphic on $|u|+|v|+|w|<1$, and the complex identity follows by analytic continuation. Moreover, \eqref{eq:covariance-joint-transform} extends to the whole domain $\mathscr D$. To see this, observe that the definition of
$\mathscr D$ implies $\Re(1+u\M_{1,0}+v\M_{0,1}+w\M_{1,1})>0$ almost surely. Both sides of \eqref{eq:covariance-joint-transform} are holomorphic on $\mathscr D$. Indeed, on every compact subset of $\mathscr D$, the real parts of the arguments of the corresponding complex powers are uniformly bounded away from zero. This yields locally uniform bounds and justifies holomorphy under the expectation and the integral. Since this domain is convex and the identity holds in a neighborhood of the origin, analytic continuation yields the claim. Differentiating at the origin, and writing
\begin{displaymath}
D_{k,m}:=\left.\partial_u^k\partial_v^k\partial_w^m\Psi_{\mathrm{cov}}^{(\alpha)}(u,v,w)\right|_{u=v=w=0},\qquad k,m\geq0,
\end{displaymath}
we obtain
\begin{equation}\label{eq:covariance-mixed-moments}
D_{k,m}=(-1)^m(\theta)_{m+2k}
\EE_{(\alpha,\theta)}\!\left[
\M_{1,0}^{k}\M_{0,1}^{k}\M_{1,1}^{m}\right].
\end{equation}
Consequently, the binomial expansion of $\CC^n$ yields
\begin{equation}\label{eq:covariance-moments}
\EE_{(\alpha,\theta)}[\CC^n]
=(-1)^n\sum_{k=0}^n\binom{n}{k}
\frac{D_{k,n-k}}{(\theta)_{n+k}},
\qquad n\geq0.
\end{equation}

Fix $0<t<1/4$. The negative binomial series and
\eqref{eq:covariance-moments}, with $m=n-k$, give
\begin{equation}\label{eq:covariance-double-series}
\begin{aligned}
\EE_{(\alpha,\theta)}[(1+t\CC)^{-r}]
&=\sum_{n=0}^{+\infty}\frac{(r)_n}{n!}(-t)^n
\EE_{(\alpha,\theta)}[\CC^n] \\
&=\frac{\Gamma(\theta)}{\Gamma(r)}
\sum_{k,m\geq0}
\frac{\Gamma(m+k+r)}{\Gamma(m+2k+\theta)}
\frac{t^{m+k}}{k!\,m!}\,D_{k,m}.
\end{aligned}
\end{equation}
This rearrangement is justified by absolute convergence. In fact, the
sum of the absolute values before rearrangement is bounded by
\begin{displaymath}
\sum_{n=0}^{+\infty}\frac{(r)_n}{n!}t^n
\EE_{(\alpha,\theta)}\!\left[
(\M_{1,1}+\M_{1,0}\M_{0,1})^n\right]
\leq\sum_{n=0}^{+\infty}\frac{(r)_n}{n!}(2t)^n<\infty.
\end{displaymath}
Since $\theta=r+1$, Euler's beta integral gives
\begin{displaymath}
\frac{\Gamma(m+k+r)}{\Gamma(m+2k+\theta)}
=\frac{1}{k!}\int_0^1
\rho^k(1-\rho)^{m+k+r-1}\,\ud\rho.
\end{displaymath}
Inserting this identity into \eqref{eq:covariance-double-series}, using
$\Gamma(\theta)/\Gamma(r)=r$, and summing over $m$, we obtain
\begin{equation}\label{eq:covariance-beta-series}
\EE_{(\alpha,\theta)}[(1+t\CC)^{-r}] = r\int_0^1(1-\rho)^{r-1} \sum_{k=0}^{+\infty}\frac{[t\rho(1-\rho)]^k}{(k!)^2}
\bigl(\partial_u^k\partial_v^k\Psi_{\mathrm{cov}}^{(\alpha)}\bigr) \bigl(0,0,t(1-\rho)\bigr)\,\ud\rho.
\end{equation}
Absolute convergence justifies interchanging the sums with the beta
integral. The Taylor series used here are locally uniformly convergent:
for $a=a_{t,\rho}$ and $w=t(1-\rho)$,
\begin{displaymath}
2a+w=2\sqrt{t\rho(1-\rho)}+t(1-\rho) \leq\sqrt t+t<1.
\end{displaymath}
Thus all evaluation points stay in a compact subset of the domain of holomorphy of $\Psi_{\mathrm{cov}}$, uniformly in $\rho\in[0,1]$ and in the angular variable introduced next.

For $2a+|w|<1$, the bivariate Taylor expansion gives
\begin{equation}\label{eq:covariance-angular-selection}
\begin{aligned}
&\frac{1}{2\pi}\int_0^{2\pi}\Psi_{\mathrm{cov}}^{(\alpha)}(ae^{i\varphi},ae^{-i\varphi},w)\,\ud\varphi\\
&\qquad=\sum_{k=0}^{+\infty}\frac{a^{2k}}{(k!)^2}\bigl(\partial_u^k\partial_v^k\Psi_{\mathrm{cov}}^{(\alpha)}\bigr)(0,0,w).
\end{aligned}
\end{equation}
Indeed, integration of $e^{i(j-k)\varphi}$ eliminates all terms with $j\ne k$. Substituting \eqref{eq:covariance-angular-selection} into \eqref{eq:covariance-beta-series} proves 
\eqref{eq:covariance-angular-transform} for $0<t<1/4$.

It remains to extend the identity to $0<t<4$. Set
\begin{displaymath}
Q(t,\rho,\varphi,\mathbf{x}) := 1+t(1-\rho)x_1x_2 +\sqrt{t\rho(1-\rho)} \bigl(e^{i\varphi}x_1+e^{-i\varphi}x_2\bigr).
\end{displaymath}
For real $0<t<4$, write $a=\sqrt{t\rho(1-\rho)}$ and
$b=t(1-\rho)$. Then, $\Re Q=1+bx_1x_2+a(x_1+x_2)\cos\varphi$. 
This expression is affine in each coordinate separately, so its minimum over $[0,1]^2$ is attained at a vertex. Its values at the vertices are $1$, $1+a\cos\varphi$ (twice), and $1+b+2a\cos\varphi$.
They are strictly positive, since
\begin{displaymath}
1+a\cos\varphi\geq1-a\geq1-\frac{\sqrt t}{2}>0
\end{displaymath}
and
\begin{displaymath}
1+b+2a\cos\varphi \geq1+t(1-\rho)-2\sqrt{t\rho(1-\rho)} =\bigl(\sqrt{t(1-\rho)}-\sqrt\rho\bigr)^2+1-\rho>0.
\end{displaymath}
In particular, for every $0<t<4$, $\rho\in[0,1]$, and $\varphi\in[0,2\pi]$, the triple $(a_{t,\rho}e^{i\varphi},a_{t,\rho}e^{-i\varphi},t(1-\rho))$ belongs to $\mathscr D$. Thus, the right-hand side of \eqref{eq:covariance-angular-transform} is well defined throughout the stated range of $t$.

Fix $t_0\in(0,4)$. By continuity and compactness, there is a complex
disc $U$ centered at $t_0$, with $\overline U\subset\{t:\Re t>0,\ |t|<4\}$,
and constants $m,M>0$ such that
\begin{displaymath}
\Re Q(t,\rho,\varphi,\mathbf{x})\geq m,
\qquad |Q(t,\rho,\varphi,\mathbf{x})|\leq M
\end{displaymath}
throughout
$\overline U\times[0,1]\times[0,2\pi]\times[0,1]^2$.
Here the square root is taken on its principal branch. If $\alpha\in(0,1)$,
the function
\begin{displaymath}
A(t,\rho,\varphi):=
\int_{[0,1]^2}Q(t,\rho,\varphi,\mathbf{x})^\alpha\,
\G(\ud\mathbf{x})
\end{displaymath}
is holomorphic in $t$ on $U$, by dominated convergence and Morera's and
Fubini's theorems. Moreover,
\begin{displaymath}
\Re A(t,\rho,\varphi)
\geq m^\alpha\cos(\pi\alpha/2)>0.
\end{displaymath}
Hence $A(t,\rho,\varphi)^{-\theta/\alpha}$ is holomorphic and uniformly
bounded on $U$, independently of $\rho$ and $\varphi$. If $\alpha=0$,
the same argument applies to the integral of $\operatorname{Log}Q$;
its exponential is holomorphic and satisfies
\begin{displaymath}
\left|\exp\left\{-\theta\int_{[0,1]^2}
\operatorname{Log}Q(t,\rho,\varphi,\mathbf{x})\,
\G(\ud\mathbf{x})\right\}\right|\leq m^{-\theta}.
\end{displaymath}
Since $(1-\rho)^{\theta-2}$ is integrable, dominated convergence and
another application of Morera's and Fubini's theorems show that the
right-hand side of \eqref{eq:covariance-angular-transform} is holomorphic
on $U$.

The left-hand side is holomorphic on $|t|<4$, because $|\CC|\leq1/4$
and, on each compact subset of that disc, $1+t\CC$ stays uniformly in
the right half-plane and away from zero. Thus both sides admit
holomorphic extensions to a connected open neighborhood of $(0,4)$.
Since they agree for $0<t<1/4$, the identity theorem proves
\eqref{eq:covariance-angular-transform} for every $0<t<4$.


\end{document}